\documentclass{amsart}
\usepackage{graphicx} 

\usepackage{amsmath,amssymb,amsthm}

\usepackage{hyperref}
\usepackage{cleveref}

\usepackage[normalem]{ulem}
\usepackage{xcolor}

\newtheorem{theorem}{Theorem}[section]
\newtheorem{lemma}{Lemma}[section]
\newtheorem{remark}{Remark}[section]
\newtheorem{corollary}{Corollary}[section]
\newtheorem{definition}{Definition}[section]
\newtheorem{proposition}{Proposition}[section]
\newtheorem{question}{Question}[section]

\DeclareMathOperator{\tr}{tr}
\DeclareMathOperator{\sizes}{\mathsf{s}}

\def\N{\mathbb{N}}
\def\Z{\mathbb{Z}}
\def\Q{\mathbb{Q}}
\def\R{\mathbb{R}}
\def\C{\mathbb{C}}

\title{On very badly approximable numbers}
\author{Zhe Cao}
\address{School of Mathematical Sciences, Nankai University, Tianjin, China}
\email{zhecao@mail.nankai.edu.cn}

\author{Harold Erazo}
\address{IMPA, Rio de Janeiro, Brazil}
\email{harold.erazo@impa.br}

\author{Carlos Gustavo Moreira}
\address{SUSTech International Center for Mathematics, Shenzhen, Guangdong, P. R. China \newline 
	IMPA, Rio de Janeiro, Brazil}
\email{gugu@impa.br}

\date{\today}

\begin{document}

	\begin{abstract}
		We prove a refined version of Markov's theorem in Diophantine approximation. More precisely, we characterize completely the set of irrationals $x$ such that $\left|x-\frac{p}{q}\right|<\frac{1}{3q^2}$ has only finitely many rational solutions: their continued fraction is eventually a balanced sequence through a simple coding. As consequence, we show that all such numbers are either quadratic surds or transcendental numbers. In particular, for any algebraic real number $x$ of degree at least $3$ there are infinitely rational numbers $\frac{p}{q}$ such that $\left|x-\frac{p}{q}\right|<\frac{1}{3q^2}$.
	\end{abstract}
	
	
	\maketitle
	
	\section{Introduction}
	
	\subsection{Main Results}\label{subsec_main_result}
	
	A longstanding open problem in Diophantine approximation is: does every algebraic real number of degree at least 
	$3$ has unbounded partial quotients in its simple continued fraction expansion?	
	This question seems to have been first posed by Khinchin \cite{Khinchin} (see also \cite{Shallit_survey}, \cite{All+2001}, \cite{Wald2004}), and it still remains unsolved.
	No algebraic real number of degree at least 3 is known with either bounded or unbounded partial quotients. 
	Bugeaud \cite{transcendence} obtained a partial result, showing that the continued fractions of such numbers can not be “too simple”, in particular, ruling out numbers whose continued fractions have linear complexity. 
	
	In this paper, we will study numbers which are very badly approximable and we will show that they have a very particular form: their continued fraction is codified by balanced sequences from certain positions and thus their continued fraction has linear complexity.
	
	For $x\in\R\setminus\Q$, define the \emph{Lagrange value} $k(x)$ by
	\begin{align*}
		k(x)=\sup\left\{ c>0 \mid \left| x - \frac{p}{q} \right| < \frac{1}{cq^2} 	\text{has infinitely many rational solutions } \frac{p}{q}\right\} \\
	\end{align*}
	The \emph{Lagrange spectrum} is define by 
	$$
	\mathcal{L}=\{ k(x)<\infty \mid x\in \R\setminus\Q \},
	$$
	while the \emph{Markov spectrum} is defined by
	$$
	\mathcal{M}=\left\{\left(\inf_{(x,y)\in\Z^2\setminus (0,0)} \left|f(x,y)\right| 	\right)^{-1} \mid f(x,y)=ax^2+bxy+cz^2, b^2-4ac=1  \right\}.
	$$
	
	A \emph{Markov number} $m$ is the largest coordinate of a solution of the \emph{Markov equation} $x^2+y^2+z^2 = 3xyz$ in the positive integers. 
	A classical result by Markov \cite{Markoff1879},\cite{Markoff1880} from 1880's states that $\mathcal{L}$ and $\mathcal{M}$ coincide before $3$ and this initial part is exactly a discrete sequence accumulating at $3$:
	$$
	\mathcal{L}\cap(-\infty,3)=\mathcal{M}\cap(-\infty,3)=\left\{ \sqrt{5},\sqrt{8},\frac{\sqrt{221}}{5},\dots,\sqrt{9-\frac{4}{m^2}},\dots \right\} ,
	$$	
	where $m$ is a Markov number.
	
	We will recall the Markov's theorem more explicitly in \Cref{subsec_markov}. 
	One consequence of Markov's theorem is that given an algebraic real number $x$ of degree at least $3$, we have $k(x)\geq 3$, that is, for any $\varepsilon>0$ the inequality $\left|x-\frac{p}{q}\right|<\frac{1}{(3-\varepsilon)q^2}$ has infinitely many solutions $p/q\in\Q$. 
	However, it says nothing about the inequality  $\left|x-\frac{p}{q}\right|<\frac{1}{3q^2}$.
	
	Inspired by Bombieri's proof of Markov's theorem \cite{Bombieri}, we characterize all irrational numbers $x$ for which the inequality $\left|x-\frac{p}{q}\right|<\frac{1}{3q^2}$ has only finitely many rational solutions $\frac{p}{q}$, using continued fractions (\Cref{thm:full_characterization}). We give a mainly combinatorial proof based on the renormalization method \cite{Bombieri} (see \Cref{subsec_renormlization} for details). As a corollary, we obtain a result on the approximation properties of algebraic real numbers of degree at least $3$ (\Cref{thm:<3_has_solution}), which is the best approximation result known so far for those numbers.

	Let $\N=\{0,1,2,\dots\}$ be the set of non-negative integers and $\N_{>0}=\{1,2,\dots\}$ the set of positive integers. We consider finite words, left infinite words, right infinite words, and bi-infinite word on a finite set, called the \emph{alphabet}. 
	We set $a=22,b=11$ and we will mainly consider alphabet $\{1,2\}$ or $\{a,b\}$ in this paper.
	When referring to the length of a word, we will always mean the length over the alphabet $\N^{*}$.
	Let $x=x_1x_2\dots x_n$ be a finite word, we define the \emph{transpose} of $x$ to be $x^T=x_nx_{n-1}\dots x_1$, moreover if $x_i\in\{a,b\}$, we define $x^{+}=x_2\dots x_n$ and $x^{-}=x_1\dots x_{n-1}$ to be the words obtained by removing the first letter and last letter, respectively.
	Similarly, we can define the transpose of an infinite word $x=x_1x_2\dots$ to be $x^T=\dots x_2x_1$ (left infinite word). 
	A sequence of finite words $\{w_i\}_{i\in\N}$ converges to a right infinite word (resp. left infinite word) $\omega$ if every prefix (resp. suffix) of $\omega$ is prefix (resp. suffix) of all but finitely many $w_i$, denoted as $\omega=\lim_{i\to\infty} w_i$.
	A sequence of finite words $\{w_i=w_i^1\vert w_i^2\}_{i\in\N}$ converges to a bi-infinite word $\underline{\omega}=\omega^1\vert\omega^2=\lim_{i\to\infty}w_i^1\vert w_i^2$ if the sequences of finite words $\{w_i^1\}_{i\in\N}$, $\{w_i^2\}_{i\in\N}$ converges to the left and right infinite word $\omega^1$,$\omega^2$ respectively.
	
	Given a word $\theta=w_1\dots w_n\in\{a,b\}^n$, we denote by $|\theta|_a$ (resp. $|\theta|_b$) the number of occurrences of the letter $a$ (resp. $b$). We say that $\theta$ is \emph{balanced} if any two subfactors $w_{i_1+1}\dots w_{i_1+k}$ and $w_{i_2+1}\dots w_{i_2+k}$ of the same length are such that $\big||w_{i_1+1}\dots w_{i_1+k}|_b-|w_{i_2+1}\dots w_{i_2+k}|_b\big|\leq 1$. The same definition applies to infinite or bi-infinite words. The balanced words were classified by Heinis \cite[Theorem 2.5]{Heinis} (see \Cref{subsec:balanced}).
	
	Let $(\alpha,\beta)$ be a pair of finite words over the alphabet $\{a,b\}$. Define the \emph{exterior renormalization operators} $\overline{U},\overline{V}$ by
	$$
	\overline{U}: (\alpha,\beta)\mapsto (\alpha\beta,\beta),\quad \overline{V}: (\alpha,\beta)\mapsto (\alpha,\alpha\beta).
	$$
	Starting with the word pair $(\alpha_0,\beta_0)=(a,b)$, we can inductively define alphabets $(\alpha_{n+1},\beta_{n+1})\in\{\overline{U}(\alpha_n,\beta_n),\overline{V}(\alpha_n,\beta_n)\}$. 
	We get an infinite binary tree, which we denote as $\overline{T}$: the root is $(a,b)$ and its children are born when we apply $\overline{U}$ and $\overline{V}$.
	Let $\overline{P}$ be the set of vertices of $\overline{T}$ and $\overline{P}_n$ be the set of vertices of $\overline{T}$ with distance exactly $n$ to the root, so in particular $(\alpha_n,\beta_n)\in\overline{P}_n$. Let $P$ denote the set of words $\alpha\beta$ for all $(\alpha,\beta)\in\overline{P}$. The words $P\cup\{a,b\}$ are known as \emph{(lower) Christoffel words}.

	The following is our main theorem that characterizes those irrationals $x=[x_0;x_1,x_2,\dots]$ such that 
	\begin{equation}\label{eq:x-pq<3}
		\left|x-\frac{p}{q}\right|<\frac{1}{3q^2}
	\end{equation} 
	has only finitely many rational solutions. Hurwitz \cite{Hurwitz} claimed that these irrationals are ``unendlich" (infinite). The first to determine their cardinality was C. Gurwood, who in his 1976 PhD thesis \cite{Gurwood} characterized those $x$ for which \eqref{eq:x-pq<3} has \emph{no} solutions and proved that the set of such $x$ is uncountable. This part of his work was cited in a small number of papers in the 1970s but does not appear to have been taken up in the subsequent literature. The result was later reproved by a different method in the undergradute thesis (in Hungarian) \cite{Harcos} of G. Harcos in 1996. Unfortunately, both Harcos and we were unaware of Gurwood’s thesis; as a result, before becoming aware of Gurwood’s earlier work, we independently re-established the cardinality in \cite{Zhe+25} by constructing uncountably many irrationals $x$ such that \eqref{eq:x-pq<3} has exactly $n$ rational solutions for every $n\in\N\cup\{\infty\}$.

	In the present work we will extend Gurwood's result by considering any finite number of solutions. Let $\lambda_{n+1}(x)=[0;x_{n},x_{n-1},\dots,x_1]+[x_{n+1};x_{n+2},x_{n+3},\dots]$. It follows form $\eqref{eq_n_th_convergent}$ and a classical result by Legendre that all the solutions of \eqref{eq:x-pq<3} are of form $\frac{p_n}{q_n}=[0;x_1,x_2,\dots,x_n]$ with $\lambda_{n+1}(x)>3$.

	\begin{theorem}\label{thm:full_characterization}
		Let $x=[x_0;x_1,x_2,\dots]$ be such that $\left| x-\frac{p}{q} \right| < \frac{1}{3q^2}$ has only finitely many solutions and let $N\in\N$ minimal such that $\lambda_N(x)\leq3$ for all $n\geq N+1$.
		
		If the continued fraction of $x$ is ultimately periodic, then there is $(\alpha,\beta)\in\overline{P}$ such that either
		\begin{itemize}
			\item $x_{N+1}x_{N+2}x_{N+3}\ldots = (\beta^T)^\infty$.
			\item $x_{N+1}x_{N+2}x_{N+3}\ldots = \beta^T(\alpha^T)^\infty$.
			\item $x_{N+1}x_{N+2}x_{N+3}\ldots = 2\alpha^{+}\alpha^\infty$.
			\item $x_{N+1}x_{N+2}x_{N+3}\ldots = 2\alpha^{+}\beta^\infty$.
		\end{itemize}
		
		Moreover, if $N\geq 1$ and $x_{N+1}x_{N+2}\dots\neq a^\infty, b^\infty$,
		then there exist unique factorization $x_{N+1}x_{N+2}\dots=\theta^Tb|a\theta R$ and is such that $[0;x_N,\dots,x_1]<[0;R]$ when $x_{N+1}=1$, and $2x_{N+1}x_{N+2}\dots=\theta^Ta|b\theta R$ and is such that $[0;x_N,\dots,x_1]<[0;2,R]$ when $x_{N+1}=2$.
		
		If the continued fraction of $x$ is not ultimately periodic, then there is a sequence of alphabets $(\alpha_{n+1},\beta_{n+1})\in \{\overline{U}(\alpha_{n},\beta_{n}),\overline{V}(\alpha_{n},\beta_{n})\}$ with both renormalization operators $\overline{U}$ and $\overline{V}$ appearing infinitely many times and such that:
		\begin{itemize}
			\item if $x_{N+1}=1$ then $x_{N+1}x_{N+2}x_{N+3}\ldots=\lim_{n\to\infty}\beta_n^T$ and if $N\geq 1$ then $[0;x_N,\dots,x_1]<[0;\lim_{n\to\infty}\alpha_n]$; or
			\item if $x_{N+1}=2$ then $x_{N+1}x_{N+2}x_{N+3}\ldots=\lim_{n\to\infty}2\alpha_n^{+}$ and if $N\geq1$ then $[0;x_N,\dots,x_1]<[0;2,\lim_{n\to\infty}\beta_n^T]$.
		\end{itemize}
		Conversely, for any such continued fraction $x=[x_0;x_1,x_2,\dots]$ the inequality $\left| x-\frac{p}{q} \right| < \frac{1}{3q^2}$ has at most $N$ solutions.
	\end{theorem}

	\begin{remark}
		The bound on the number of solutions can not be improved, for example consider the continued fraction given by $[0;x_1,\dots,x_N,\lim_{n\to\infty}\beta_n^T]$ or $[0;x_1,..,x_N,\beta^\infty]$ where $x_1,\dots,x_N\geq 3$.
	\end{remark}
	
	\begin{remark}
		The condition $[0;x_N,...,x_1]<[0;R]$ and $[0;x_N,...,x_1]<[0;2,R]$ for the eventually periodic case is to ensure that there are no more bad cuts after the $N$ position. 
		For example if $\alpha=aabab$, then the infinite word $\alpha^\infty=aababaabab\dots$ has a unique factorization $aa|bab\dots=\theta^Ta|b\theta R$, so $\lambda_n(x)<3$ for that position if and only if $[0;x_N,...,x_1]<[0;2,R]$. 
		On the other hand, the non-periodic cases always has infinitely many such factorizations, so we have to impose a different condition.
	\end{remark}

	\begin{remark}\label{rmk_tildem_<3_is_finite}
		A consequence of the \Cref{thm:full_characterization} is that, given $\varepsilon>0$, the set of $x$ such that $\left|x-\frac{p}{q}\right|\geq \frac{1}{(3-\varepsilon)q^2}$ for all $p/q\in\Q$ is finite. On the other hand, it also follows from the \Cref{thm:full_characterization} that the set of irrationals $x$ with $k(x)=3$ and such that $\left|x-\frac{p}{q}\right|\geq \frac{1}{3q^2}$ for all $p/q\in\Q$, is a Cantor set of Hausdorff dimension zero.
	\end{remark} 
	
	The previous theorem may be viewed as a refinement of Markov’s theorem from the perspective of Diophantine approximation, as its consequences are new even in the ultimately periodic case. Indeed, if $k(x)<3$ then the continued fraction expansion of $x$ eventually coincides with a periodic expansion of the form $(\alpha\beta)^\infty$ for some $(\alpha,\beta)\in\overline{P}$, or with $a^\infty$ or $b^\infty$. We provide an explicit criterion for estimating where the period appears, namely the last index $N$ for which $\lambda_N(x)>3$. Moreover, the sequence $x_{N+1}x_{N+2}\dots$ when $x_{N+1}=1$ and $2x_{N+1}x_{N+2}\dots$ when $x_{N+2}$ will be a balanced sequence.
	
	Our proof is essentially self-contained; in particular, the Diophantine approximation part of Markov’s theorem (\Cref{thm_markov}) follows directly from \Cref{thm:full_characterization}. Gurwood’s proof relies on a characterization of lower and upper balanced sequences; however, a key step in that argument (\Cref{lem:not_generalizable_lemma}) does not generalize, which prevents us from using his elegant ideas to obtain a complete proof of \Cref{thm:full_characterization}. Nevertheless, we present Gurwood’s proof in detail and show that his result is equivalent to the case $N=0$ of \Cref{thm:full_characterization}, after applying results from the theory of balanced words.
	
	Together with a deep result by Bugeaud \cite[Theorem 1.1]{transcendence}, we can deduce the following theorem.
	
	\begin{theorem}\label{thm:equal3_is_tran}
		Given $x\in\R\setminus\Q$ with $k(x)=3$ and such $\left| x-\frac{p}{q} \right| < \frac{1}{3q^2}$ has only finitely many rational solutions, then $x$ transcendental.
	\end{theorem}
	
	As a corollary we get
	
	\begin{theorem}\label{thm:<3_has_solution}
		If $x$ is an algebraic real number of degree at least $3$, then the inequality $\left|x-\frac{p}{q}\right|<\frac{1}{3q^2}$ has infinitely many rational solutions $p/q$.
	\end{theorem}
	
	It is much more difficult to obtain a similar characterization for all $k^{-1}(3)$.
	In other words, we do not know the answer of the following question:
	
	\begin{question}
		Is there any algebraic real number in the set $k^{-1}(3)$?
	\end{question}
	
	The reason is because for example numbers such as $[0;1^{s_1},2,2,1^{s_2},2,2,1^{s_3},2,2,\dots]$ where $(s_n)_{n_\geq 1}$ is an increasing sequence of positive integers belong to $k^{-1}(3)$. 
	It is known (see \cite[Example 4.5.17]{wordcombinatorics}) that if $s_n\sim r^n$ for some $r>1$, then the complexity is linear (so the number is trascendental) but if $s_n\sim n^r$ for some $r\in\N_{>0}$, then the complexity is of the order $n^{1+1/r}$, so Bugeaud theorem does not apply.
	On the other hand, $k^{-1}(3)$ do not contain rationals or quadratic numbers, since for rational $x$ the inequality $|x-\frac{p}{q}|<\frac{1}{cq^2}$ could only have finitely many rational solutions $\frac{p}{q}$,
	and it follows from Markov's theorem that irrationals with $k(x)=3$ are not eventually periodic, equivalently, not quadratic.

	\subsection{Spectrum \texorpdfstring{$\widetilde{\mathcal{M}}$}{M-tilde}}\label{subsec_tildeM}

	Analogous with the Lagrange value, we can define, for all $x\in\R$,
	\begin{equation*}
		\widetilde{m}(x) := \inf\left\{c>0:\left|x-\frac{p}{q}\right|<\frac{1}{cq^2} \text{ 	has no rational solutions } \frac{p}{q}\neq x\right\},
	\end{equation*}
	and define the spectrum
	\begin{equation}\label{eq_def_tildeM}
		\widetilde{\mathcal{M}}:=\{\widetilde{m}(x)<\infty\mid x\in\R\setminus\Q\}.
	\end{equation}
	
	Notice that $\widetilde{m}(x)$ is the best constant $c$ such that $\left|x-\frac{p}{q}\right|\geq\frac{1}{cq^2}$ for all $\frac{p}{q}\in\Q\setminus\{x\}$. In particular it always holds that 
	\begin{equation}\label{eq:basic_inequality}
		\left|x-\frac{p}{q}\right|\geq \frac{1}{\widetilde{m}(x)q^2} \qquad \text{for all } \frac{p}{q}\in\Q\setminus\{x\}.
	\end{equation}
	
	For rational numbers $x=\frac{a}{b}$, the value of $\widetilde{m}(x)$ is always attained, because for $\frac{p}{q}\neq\frac{a}{b}$ with $q> b$ one has $q^2\cdot\left|\frac{a}{b}-\frac{p}{q}\right|=\frac{q}{b}\cdot|aq-bp|\geq \frac{q}{b}> 1$, while if we chose $p=\lfloor \frac{a}{b}\rfloor$ and $q=1$, one clearly has $q^2\cdot\left|\frac{a}{b}-\frac{p}{q}\right|=\left|\frac{a}{b}-\lfloor \frac{a}{b}\rfloor\right|\leq 1$. 
	
	The real numbers $x$ such that $\widetilde{m}(x)\leq 3$ are called \emph{poorly approximable numbers}. They were classified for rational $x$ by Flahive \cite{Flahive} and for irrational $x$ by Gurwood \cite{Gurwood}. The irrationals $x$ with $\tilde{m}(x)<3$ were rediscovered independently by Florek in an unpublished preprint \cite{Florek} using Markov quadratic forms. Recently, the rationals $x$ with $\widetilde{m}(x)\leq 3$ were classified by Springborn \cite{Springborn} through a hyperbolic geometric correspondence. 
	
	In this paper we will only consider irrational $x$, so we will always assume that $\widetilde{m}$ is restricted to the irrationals. Similar with Markov's theorem, one can characterize $\widetilde{\mathcal{M}}$ before 3 and in fact, show, that this beginning is homothetic with the beginning of Markov's spectrum.
	\begin{theorem}\label{thm:characterization_tildeM}
		We have that
		\begin{equation*}
			\widetilde{\mathcal{M}}\cap(-\infty,3)=\left\{\frac{3+\mu}{2}: \mu\in\mathcal{M}\cap(0,3)\right\}.
		\end{equation*}
		In particular from Markov's theorem
		\begin{equation*}
			\widetilde{\mathcal{M}}\cap(-\infty,3)=\left\{\frac{3+\sqrt{9-\frac{4}{m^2}}}{2} \mid m \text{ a Markov number }\right\}.
		\end{equation*}
	\end{theorem}

	Although the previous theorem was already known \cite{Gurwood, Flahive, Florek, Harcos}, we offer yet another proof in \Cref{sec:tildeMbefore3} based on renormalization. Notice that the case $N=0$ of \Cref{thm:full_characterization} gives a full characterization of the irrationals $x$ such that $\left|x-\frac{p}{q}\right|\geq\frac{1}{3q^2}$ for all rational $p/q$, which is precisely the set
	\begin{align*}
		\widetilde{m}^{-1}\left(\widetilde{\mathcal{M}}\cap(-\infty,3]\right) &=\left\{x\in\R\setminus\Q\mid \widetilde{m}(x)\leq 3\right\} \\
		&=\left\{x\in\R\setminus\Q\mid \left|x-\frac{p}{q}\right|\geq\frac{1}{3q^2}\text{ for all rational }\frac{p}{q}\right\}.
	\end{align*}
	As proved by Harcos and Florek using a quadratic forms approach, these irrationals are closely related to the roots of Markov forms. In \Cref{subsec:spectrum_widetilde_M_before_3} we show that the following theorem is a direct corollary of the case $N=0$ of \Cref{thm:full_characterization}. Note that the number of solutions of \eqref{eq:x-pq<3} remains the same if we change $x$ by $k\pm x$ for any $k\in\Z$; in particular $\widetilde{m}(x)=\widetilde{m}(x+k)$.
	
	\begin{theorem}\label{thm:theorem_Florek_Harcos}
		Let $x$ be an irrational. Then $\widetilde{m}(x)<3$ if and only if $x=k\pm\theta_r$ or $x=k\pm\Theta_r$ for some $k\in\Z$, where $\theta_r,\Theta_r$ are the roots of some Markov form $F_r$.
	\end{theorem}
	
	Consider the continued fraction of $x\in\R\setminus\Q$
	$$
	x=[x_0;x_1,x_2\dots]=x_0+\frac{1}{x_1+\frac{1}{x_2+\genfrac{}{}{0pt}{0}{}{\ddots}}},
	$$
	let $\frac{p_n}{q_n}=[x_0;x_1,\dots,x_n]$ be the \emph{$n$-th convergent} of $x$, we have
	\begin{equation}\label{eq_n_th_convergent}
		x-\frac{p_n}{q_n} = \frac{(-1)^{n}}{(\gamma_{n+1}+\eta_{n+1})q_n^2}=\frac{(-1)^{n}}{\lambda_{n+1}(x)q_n^2}
	\end{equation}
	where $\gamma_{n+1}=[x_{n+1};x_{n+2},x_{n+3},\dots]$, $\eta_{n+1}=[0;x_{n},x_{n-1},\dots,x_1]$ and $\lambda_{n+1}(x)=\gamma_{n+1}+\eta_{n+1}$. 
	A classical result by Legendre states that if $\left| x-\frac{p}{q}\right|<\frac{1}{2q^2}$ , where $\frac{p}{q}\in\Q$, then $\frac{p}{q}=\frac{p_n}{q_n}$ for some $n\in\N$.

	By the results above, we have an expression of the Lagrange value of $x$:
	\begin{equation}\label{defk}
		k(x)=\limsup_{\frac{p}{q}\in\Q, \frac{p}{q}\to x}\frac{1}{q|p-qx|}=\limsup_{n\to\infty}\left(\gamma_{n+1}+\eta_{n+1}\right)=\limsup_{n\to\infty}\lambda_{n+1}(x).  
	\end{equation}
	Analogously
	\begin{equation}\label{def_tildem1} 
		\widetilde{m}(x) = \sup_{\frac{p}{q}\in\Q\setminus\{x\}}\frac{1}{q|p-qx|}=\sup_{n\in\N}\left(\gamma_{n+1}+\eta_{n+1}\right)=\sup_{n\in\N}\lambda_{n+1}(x),
	\end{equation}
	so we always have $\widetilde{m}(x)\geq k(x)$.
	In particular $\widetilde{\mathcal{M}}$ is defined in a similar way as $\mathcal{L}$ and $\mathcal{M}$, but is different and its structure is interesting since it is very related to them.
	It is easy to see that is different, because for example its first point is $\widetilde{m}(\frac{\sqrt{5}-1}{2}=[0;1,1,\dots])=\frac{3+\sqrt{5}}{2}$,
	while by Markov's theorem $\frac{3+\sqrt{5}}{2}\notin\mathcal{L}$ or $\mathcal{M}$.

	Notice that if $\widetilde{m}(x)$ is rational, then necessarily $\widetilde{m}(x)=k(x)$, because each term $\lambda_{n+1}(x)$ in \eqref{def_tildem1} is irrational. In particular $\widetilde{m}^{-1}(3)\subseteq k^{-1}(3)$. On the other hand, it holds that $\{x\in\R\setminus\Q:\widetilde{m}(x)<3\}\subset\{x\in\R\setminus\Q:k(x)<3\}$, which is countable by Markov's theorem.
	Actually, the set $\{x\in\R\setminus\Q:\widetilde{m}(x)<3-\varepsilon\}$ is finite for any $\varepsilon>0$ (see \Cref{rmk_tildem_<3_is_finite}). Another consequence of \Cref{thm:full_characterization}, is that for $t<3$, $t\neq\frac{3+\sqrt{5}}{2}$, $t\neq\frac{3+2\sqrt{2}}{2}$ the cardinality of the set $\widetilde{m}^{-1}(t)\cap(0,1)$ is always a multiple of $4$  and if the Frobenius Uniqueness conjecture is true, then this cardinality will always be $0$ or $4$. For the exceptional values it holds that $\widetilde{m}^{-1}(\frac{3+\sqrt{5}}{2})=\{\frac{\sqrt{5}-1}{2},\frac{3+\sqrt{5}}{2}\}$ and $\widetilde{m}^{-1}(\frac{3+2\sqrt{2}}{2})=\{\sqrt{2}-1,2-\sqrt{2}\}$.
	
	From \eqref{eq:basic_inequality} it follows that given any $t\in\R$, the set $\widetilde{m}^{-1}((-\infty,t])$ is closed. In contrast, recall that $k^{-1}(t)$ is dense in $\mathbb{R}$ for any $t$ in the Lagrange spectrum. 
	Another consequence of \Cref{thm:full_characterization} is that the set $\widetilde{m}^{-1}(t)$ is closed for all $t\leq 3$. 
	However, this is not true in general: for example the irrationals $z_n=[0,1^{2n},2,1^{2n+2},2,1^{2n+4},2,\dots]$ are such that $\widetilde{m}(z_n)=1+\sqrt{5}$ for all $n\geq 1$, but $z_n$ tends to $z=[0;1,1,\dots]=\frac{\sqrt{5}-1}{2}$ which satisfies $\widetilde{m}(z)=\frac{3+\sqrt{5}}{2}$. From this same construction it is not difficult to see that $\widetilde{M}$ has positive Hausdorff dimension after $3$.

	Probably the first to consider the kind of problem \eqref{eq:basic_inequality} was Davenport, who in 1947 asked which is the largest constant $c$ such that for any irrational number $x$, the inequality $|x-\frac{p}{q}|<\frac{1}{cq^2}$ has always at least one solution. 
	In 1948, Prasad \cite{Prasad1948} computed that this constant is $\frac{3+\sqrt{5}}{2}$ and is optimal for any number of the form $k\pm\frac{1+\sqrt{5}}{2}$, $k\in\Z$. Excluding these numbers, the best constant was computed by Eggan \cite{Eggan} is $\frac{3+2\sqrt{2}}{2}$ and is optimal for numbers of the form $k\pm (1+\sqrt{2})$, $k\in\Z$. The next constant was computed in \cite{PrasadPrasad}. In \cite{Bur+2002}, the authors go much further by characterizing the $r$-th best constants $C_r(n)$ ($r\geq1$ and $n\geq 1$ fixed) such that the inequality $\left|x-\frac{p}{q}\right|\geq\frac{1}{C_r(n)q^2}$ has always at least $n$ solutions among all $x$ that are not $GL_2(\Z)$ equivalent to all the $\gamma_{i}$ with corresponding $m_i<m_{(r)}$, where $\theta_i,m_i$ and $m_{(r)}$ are defined in \Cref{subsec_markov}.

	This article is organized as follows. In \Cref{subsec_markov} we recall Markov's theorem. In \Cref{subsec_renormlization} we introduce the key notion of renormalization and recall several properties of Christoffel words. 
	In \Cref{subsec:roots} we review the properties of roots of Markov forms and deduce \Cref{thm:theorem_Florek_Harcos}. In \Cref{subsec:balanced} we recall several lemmas from Gurwood’s work on balanced sequences and introduce new ones. \Cref{sec:characterization} is devoted to the proof of \Cref{thm:full_characterization}. In \Cref{subsubsec:gurwood} we recall Gurwood’s proof and show that it is equivalent to the case $N=0$ of \Cref{thm:full_characterization}. In \Cref{sec:tildeMbefore3} we prove \Cref{thm:characterization_tildeM}. Finally, in \Cref{sec:remarks} we discuss a similar spectrum to $\widetilde{M}$ and pose a question concerning the approximation of complex numbers.

	\subsection{Acknowledgments}
	The first author would like to thank the Instituto de Matemática Pura e Aplicada (IMPA) where this work began. The second author was partially supported by CAPES and FAPERJ. The third author was partially supported by CNPq and FAPERJ. We are grateful to Davi Lima for mentioning the spectrum defined by Divis.

	\section{Preliminary}
	
	\subsection{Periodic continued fractions}
	
	Given a finite word $w=a_0\dots a_n$ where $a_i$ are positive integers, define the quadratic irrational $\gamma$ by
	\begin{equation*}
		\gamma=[w,\eta]=[w,w,\dots].
	\end{equation*}
	Define the matrix
	\begin{equation}\label{eq:matrix_M_w}
		M_w=\begin{pmatrix}
			a_0 & 1 \\
			1 & 0
		\end{pmatrix}
		\begin{pmatrix}
			a_1 & 1 \\
			1 & 0
		\end{pmatrix}\dotsb
		\begin{pmatrix}
			a_n & 1 \\
			1 & 0
		\end{pmatrix}=\begin{pmatrix}
			p & p^\prime \\
			q & q^\prime
		\end{pmatrix}
	\end{equation}
	so we have the convergents $p/q=[a_0;\dots,a_n]$ and $p^\prime/q^\prime=[a_0;\dots,a_{n-1}]$. Moreover, one has that
	\begin{equation*}
		\gamma=\frac{p\gamma+p^\prime}{q\gamma+q^\prime}
	\end{equation*}
	so
	\begin{equation}\label{eq:eta_formula}
		\gamma=\frac{p-q^\prime+\sqrt{\Delta}}{2q}
	\end{equation}
	where $\Delta=\tr(M_w)^2+4\cdot(-1)^{n}$ is the discriminant.
	
	A classical theorem by Galois says that if a quadratic irrational $\gamma=\frac{P+\sqrt{D}}{Q}$ has a purely periodic continued fraction $\gamma=[\overline{a_0;a_1,\dots,a_m}]$, then its Galois conjugate $\overline{\gamma}=\frac{P-\sqrt{D}}{Q}$ satisfies $-\frac{1}{\overline{\gamma}}=[\overline{a_m;a_{m-1},\dots,a_0}]$.

	We have in general that if $w=a_0\dots a_n$ then
	\begin{equation*}
		[\overline{a_0,\dots,a_n}]+[0;\overline{a_n,\dots,a_0}]=\gamma+(-\overline{\gamma})=\frac{\sqrt{\Delta}}{q(w)}
	\end{equation*}
	where $q(w)$ is the lower-left corner of the matrix $M_w$ and $\Delta=\tr(M_w)^2-4$ is the discriminant.

	\subsection{Markov's Theorem}\label{subsec_markov}
	
	There are essentially three approaches for proving Markov's theorem. The first is based on the theory of continued fractions \cite{CF89, Bombieri}, the second on the theory of indefinite binary quadratic forms \cite{dickson1930studies, Cassels} and the third on hyperbolic geometry \cite{CSeries, SpringbornMarkov}. The approach adopted in this paper falls within the first paradigm.
	
	\subsubsection{Badly approximable numbers}    
	
	Given a matrix $A=
	\left(\begin{smallmatrix} a & b \\ c & d \end{smallmatrix}\right)
	\in GL_2(\Z)$, that is, $a,b,c,d\in\Z$ and $|ad-bc|=1$, define the action of $A$ on $x\in\R\setminus\Q$ by $Ax=\frac{ax+b}{cx+d}$.
	We say that $y\in\R\setminus\Q$ is \emph{$GL_2(\Z)$-equivalent} to $x$ if there exist a $A\in GL_2(\Z)$ such that $y=Ax$.
	
	We have a equivalent description of $GL_2(\Z)$-equivalent by continued fraction:
	$x,y\in\R\setminus\Q$ are $GL_2(\Z)$-equivalent if and only if their continued fractions eventually coincide.
	
	Recall that a Markov number is the largest coordinate of a positive integer solution triple $(x,y,z)$ of the Markov equation.
	The solution $(x,y,z)$ is called normalized if $x\leq y\leq z$ .
	The \emph{multiplicity} of a Markov number $m$ is the number of distinct normalized solutions $(x,y,z)$ of Markov equation with $z=m$.
	
	Let $(m_r)_{r\geq1}$ be the sequence of the ordered Markov numbers with multiplicity, i.e., a non-decreasing sequence whose terms are all Markov numbers and such that the number of times each term appears is equal to its multiplicity.
	We can enumerate the set of normalized solutions of the Markov equation by a sequence of pairwise distinct triples $\{(x_r, y_r, z_r)\}_{r \geq 1}$ with $z_r = m_r$. 
	On the other hand, let $(m_{(r)})_{r\geq 1}$ be the sequence of the ordered Markov numbers without multiplicity, i.e., $m_{(r)}$ is the $r$-th largest element of the set of Markov numbers. 
	
	The well known uniqueness conjecture by Frobenius states that all the Markov numbers have multiplicity of $1$, that is, $m_r=m_{(r)}$ for every $r\geq 1$. The conjecture still remains open.
	
	The following result by Markov \cite{Markoff1879, Markoff1880} gives a full characterization of $\mathcal{L}$ and $\mathcal{M}$ before $3$, and also, of those badly approximable numbers with constant $k(x)<3$.
	
	\begin{theorem}\label{thm_markov}{\rm$($Markov$)$} 
		$$
		\mathcal{L}\cap(-\infty,3)=\mathcal{M}\cap(-\infty,3)=\left\{ \sqrt{5},\sqrt{8},\frac{\sqrt{221}}{5},\dots,\frac{\sqrt{9m_{(r)}^2-4}}{m_{(r)}},\dots \right\}.
		$$
		
		Moreover, there is a sequence of $GL_2(\Z)$-inequivalent irrationals $\gamma_{r}=\frac{a_r+\sqrt{9{m_r}^2-4}}{b_r}$ $(a_r,b_r\in\Z)$ such that $k(\gamma_{r})=\frac{\sqrt{9m_r^2-4}}{m_r}$, and every $x\in\R\setminus\Q$ with $k(x)<3$ is $GL_2(\Z)$-equivalent to some $\gamma_{r}$.
		
	\end{theorem}
	
	We can further assume the irrationals $\gamma_{r}$ above to be purely periodic and describe them through the alphabets.
	Let $\gamma_{1}=[b,b,\dots]=[1,1,\dots]=\frac{1+\sqrt{5}}{2}$ and $\gamma_{2}=[a,a,\dots]=[2,2,\dots]=1+\sqrt{2}$. 
	It follows from \cite[Theorem 26]{Bombieri} that there exists a surjective map from $\{(x_r,y_r,z_r)\}_{r\geq3}$ onto $\overline{P}$.
	So given $r\geq 3$, take the image $(\alpha_r,\beta_r)\in\overline{P}$ of $(x_r,y_r,z_r)$ and let
	\begin{equation}\label{eq:gamma_r}
		\gamma_{r}=[(\alpha_r\beta_r)^\infty]=[\alpha_r,\beta_r,\alpha_r,\beta_r,\dots]\in\Q\left(\sqrt{9m_r^2-4}\right).
	\end{equation}
	The numbers $\gamma_r$ are $GL_2(\Z)$-inequivalent between each other because of \cite[Theorem 18]{Bombieri}. 
	
	\subsubsection{Markov quadratic forms}\label{subsubsec:markov_forms}
	
	We will consider now binary quadratic forms with real coefficients. Given two binary quadratic forms $f(x,y)$ and $g(x,y)$, we say that they are \emph{$GL_2(\Z)$-equivalent} if there is $A=
	\left(\begin{smallmatrix} a & b \\ c & d \end{smallmatrix}\right)
	\in GL_2(\Z)$ such that $f(ax+by,cx+dy)=g(x,y)$. Recall that a quadratic form $f(x,y)=Ax^2+Bxy+Cy^2$ is indefinite (i.e., takes positive and negative values) if and only if is discriminant $B^2-4AC$ is positive.
	
	Suppose that $(x_r,y_r,m_r)$ is a positive integer solution of Markov's equation with $\max\{x_r, y_r\}\leq m_r$. Let $0\leq k_r<m$ and $l_r>0$ be the integers defined by
	\begin{equation}\label{eq:k-l}
		k_r\equiv\frac{y_r}{x_r}\equiv-\frac{x_r}{y_r}\pmod{m_r} \qquad k_r^2+1=l_rm_r.
	\end{equation}
	We order $x_r$ and $y_r$ such that $k_r$ is minimal. In particular $0\leq 2k_r\leq m_r$.
	
	\begin{definition}
		Let $m_r$ be a Markov number. Then the \emph{Markov form} $F_{r}$ is the binary quadratic form 
		\begin{equation*}
			m_rF_{r}=m_rx^2+(3m_r-2k_r)xy+(l_r-3k_r)y^2
		\end{equation*}
		where $k_r$ and $l_r$ are defined on \eqref{eq:k-l}.
	\end{definition}
	
	The definition of $k_r$ and $l_r$ is asymmetric on $x_r,y_r$. If we exchange $x_r$ with $y_r$, the corresponding $k_r^\prime, l_r^\prime$ defined on \eqref{eq:k-l} yield the form $F_{r}^\prime(x,y)=F_{r}(x+2y,-y)$.

	The second part of Markov's theorem \cite{Markoff1879, Markoff1880} is about minima of indefinite binary quadratic forms. We refer the reader to the classical books \cite{Cassels} and \cite{dickson1930studies} for this part and also \cite[Appendix C]{Lim+21}. Given a quadratic form $f(x,y)=Ax^2+Bxy+Cy^2$ define its Markov value
	\begin{equation*}
		m(f)=\frac{\sqrt{B^2-4AC}}{\inf_{(x,y)\in\Z^2\setminus(0,0)}|f(x,y)|}
	\end{equation*}
	\begin{theorem}[Markov]
		Let $f(x,y)$ be an indefinite binary quadratic form. Then:
		\begin{enumerate}
			\item The inequality $m(f)<3$ holds if and only if $f$ is $GL_2(\Z)$--equivalent to a multiple of some Markof form.
			\item The equality $m(f)=3$ holds if and only if $f$ is $GL_2(\Z)$--equivalent to a multiple of a form $F(x,y)=A^\prime(x-\theta y)(x-\Theta y)$ where the roots $\theta>0>\Theta$, $\theta=[a_0;a_1,\dots]$ and $-1/\Theta=[0;a_{-1},a_{-2},\dots]$ are such that $(a_n)_{n\in\Z}$ is a bi-infinite Sturmian sequence (see \Cref{subsec:balanced}) after replacing $(2,2)\mapsto a$ and $(1,1)\mapsto b$.
		\end{enumerate}
	\end{theorem}
	
	We refer readers to expository article by Bombieri \cite{Bombieri} and to \cite{CF89, Lim+21, Cassels} for more details about these theorems, the uniqueness conjecture, and results about the structure of $\mathcal{L}$ and $\mathcal{M}$.

	\subsection{Dynamical Characterization}
	
	In 1921, Perron \cite{perron1921} gave a dynamical system characterization of Lagrange spectrum and Markov spectrum.
	Consider the symbolic dynamical system $(\Sigma,\sigma)$, where $\Sigma=(\N^{*})^{\Z}$ is the symbolic space and $\sigma: \Sigma\rightarrow\Sigma, (x_i)_{i\in\Z}\mapsto (x_{i+1})_{i\in\Z}$ is the left shift. 
	For any $\underline{x}=(x_i)_{i\in\Z}\in\Sigma$ , define the \emph{height function} $\lambda$ on $\Sigma$ as
	$$
	\lambda(\underline{x})=[x_0;x_1,x_2,\dots]+[0;x_{-1},x_{-2},\dots]
	$$
	and let $\lambda_n(\underline{x})=\lambda(\sigma^{n}(\underline{x}))=[x_n;x_{n+1},x_{n+2},\dots]+[0;x_{n-1},x_{n-2},\dots]$
	and we extend the definition of the Lagrange value to $\underline{x}$ by
	$$
	l(\underline{x})=\limsup_{n\to\infty}{\lambda_n(\underline{x})}.
	$$
	If we define a Lagrange spectrum through bi-infinite sequences, it will give the same set, that is
	$$
	\mathcal{L}=\left\{l(\underline{x})<\infty \mid \underline{x}\in\Sigma \right\}.
	$$
	This is because the value of limsup is preserved under projection. 
	More precisely, let
	$$
	\pi: \Sigma=(\N^{*})^{\Z} \rightarrow \Sigma^{+}=(\N^{*})^{\N^{*}},\quad \underline{x}=(x_i)_{i\in\Z}\mapsto \pi(\underline{x})=(x_i)_{i\in\N^{*}}
	$$
	be the projection of a bi-infinite word to the right part, we have $l(\underline{x})=k(\pi(\underline{x}))$.
	
	Similarly, let $m(\underline{x})=\sup_{n\in\Z}\lambda_n(\underline{x}) $ be the \emph{Markov value} of $\underline{x}$
	and we will have
	$$
	\mathcal{M}=\left\{m(\underline{x})<\infty \mid \underline{x}\in\Sigma \right\}.
	$$
	This is so called dynamical characterization of Lagrange and Markov spectrum.
	
	In general we do not have $m(\underline{x})=\widetilde{m}(\pi(\underline{x}))$, since for example we have $m(\dots 111\dots)=\sqrt{5}$ while $\widetilde{m}([0;1,1,\dots])=\frac{3+\sqrt{5}}{2}$.
	The definitions of $m(\underline{x})$ and $\widetilde{m}(x)$ are very similar and this is also the reason we use the notations $\widetilde{m}$ and $\widetilde{\mathcal{M}}$.

	\subsection{Renormalization}\label{subsec_renormlization}
	
	The main tool we used in this paper is renormalization, as introduced in \cite{Bombieri}, where this technique was used to give a full characterization of those bi-infinite sequences with Markov value at most 3.

	Recall that we define the exterior renormalization operators $\overline{U},\overline{V}$ and the infinite binary tree $\overline{T}$ in \Cref{subsec_main_result}.
	The following proposition is \cite[Proposition 2.16]{Zhe+25}, so called renormalization algorithm, and is originally from Bombieri's work \cite{Bombieri}. We will not use it, although it is instructive to understand the similarities between the spectra.
	
	\begin{proposition}\label{prop_renorm}
		A bi-infinite word $\omega$ has Markov value smaller than $3$ if and only if it can be written as a constant word in some alphabet $(\alpha,\beta)\in \overline{P}$, and $\omega$ has Markov value exactly $3$ if and only if for all $n\in\N$, the word $\omega$ can be written as a non-constant word for some alphabet $(\alpha_n,\beta_n)\in\overline{P}_n$.
	\end{proposition}
	
	Let $U,V$ be the Nielsen substitutions of the free group $\langle a,b \rangle$, which are
	\begin{align*}
		U : 
		\begin{aligned}
			&a\mapsto ab \\
			&b\mapsto b  \\
		\end{aligned}
		,\quad
		V : 
		\begin{aligned}
			&a\mapsto a \\
			&b\mapsto ab.  \\
		\end{aligned}
	\end{align*}
	The actions of $U,V$ can be naturally extended to word pairs over the alphabet $\{a,b\}$ by
	\begin{equation*}
		U:(\alpha,\beta) \mapsto (U(\alpha),U(\beta)),\quad V(\alpha,\beta) \mapsto (V(\alpha),V(\beta)).    
	\end{equation*}
	
	We call the operators $U,V$ by \emph{inner renormalization operators}.
	Renormalization by $U,V$ or $\overline{U},\overline{V}$ are intrinsically related by \cite[Proposition 2.16]{Zhe+25} (see also \cite[Theorem 2.6.1]{Reutenauerbook}).
	
	\begin{proposition}\label{relbtrenorm}
		Let $R_1,R_2,\dots R_n\in\{U,V\}$ be interior renormalization operators. We have the following equality
		$$
		(\overline{R}_1\dots\overline{R}_n)(a,b)=(R_n\dots R_1)(a,b).
		$$
	\end{proposition}
	
	\subsection{Some properties of Christoffel words}\label{subsec:properties}
	
	In this section, we present some useful lemmas.
	We will make use of the following identities.
	
	\begin{lemma}\label{lem:symmetry1}
		Let $(\alpha_n,\beta_n)\in \overline{P}_n$. We have
		\begin{enumerate}
			\item $\alpha_n\beta_n=a\theta b$ with $\theta$ palindrome. \label{item1}
			\item $\alpha_n\alpha_n\beta_n\beta_n=a\theta ab\theta b$ with $\theta$ palindrome. \label{item4}
			\item There exists $W_n=R_1R_2\dots R_n$ with $R_i\in\{U,V\}$ such that $\alpha_n=W_n(a)$ and $\beta_n=W_n(b)$. \label{item2}
		\end{enumerate} 
	\end{lemma}
	
	\begin{proof}
		(\ref{item1}) was done in the proof of \cite[Theorem 15]{Bombieri}, see also \cite[Remark 3.9]{EGRS2024}.
		(\ref{item2}) corresponds to \cite[Lemma 3.7]{EGRS2024}.
		(\ref{item4}) was contained in the proof of \cite[Lemma 3.15]{EGRS2024} (as a special case).
	\end{proof}
	
	The following lemma is from the proof of \cite[Theorem 15]{Bombieri}, see also \cite[Lemma 3.14]{EGRS2024}.
	
	\begin{lemma}\label{lem:commutative_identities}
		For any finite word $w$ over the alphabet $\{a, b\}$, we have the identities $bU(w^T) = U(w)^Tb$ and $V(w^T)a = aV(w)^T$. In particular, if $w$ is a palindrome, then $b U(w)$ and $V(w) a$ are palindromes as well.
	\end{lemma}
	
	The following lemma is \cite[Lemma 3.8]{EGRS2024}, see also \cite[Lemma 2.18]{Zhe+25}.
	
	\begin{lemma}\label{lem:symmetry2}
		For every $(\alpha,\beta)\in \overline{P}$, $\alpha$ starts with $a$ and $\beta$ ends with $b$.
		We always have the equalities $\alpha\beta=\beta_a\alpha^b$, $\alpha^b\beta=\beta^T\alpha^b$, $\alpha\beta_a=\beta_a\alpha^T$, $\alpha^b_a=\alpha^T$ and $\beta^b_a=\beta^T$.
		As a consequence, words $\alpha^k\beta,\alpha\beta^k$, with $k\geq 1$, start with $\beta_a$ and end with $\alpha^b$.
	\end{lemma}
	
	\begin{corollary}
		For every $(\alpha,\beta)\in\overline{P}$ we have
		\begin{equation}\label{eq:infinite_identities}
			\alpha^\infty=\alpha\alpha\dots=\beta_a(\alpha^T)^\infty, \qquad \beta^\infty=\dots\beta\beta=(\beta^T)^\infty\alpha^b.
		\end{equation}
	\end{corollary}
	
	\begin{proof}
		By \Cref{lem:symmetry2} we have $\beta_a(\alpha^T)^k = \alpha^k\beta_a$ and $(\beta^T)\alpha^b=\alpha^b\beta^k$, letting $k\to\infty$ shows the identities.
	\end{proof}
	
	For more background material we refer the reader to the book \cite{Reutenauerbook}.

	\subsection{Roots of Markov quadratic forms}\label{subsec:roots}
	Since we can factor
	\begin{equation*}
		F_{r}(x,y)=\left( x+\frac{3m_r - 2k_r}{2m_r} y \right)^2-\left( \frac{9}{4} - \frac{1}{m_r^2} \right) y^2,
	\end{equation*}
	letting
	\begin{equation*}
		\Delta = \frac{3+\sqrt{9 - \frac{4}{m_r^{2}}}}{2},
	\end{equation*}
	we have that $F_r(x,y)=(x-\theta_ry)(x-\Theta_ry)$, where the roots are
	\begin{equation*}
		\theta_{r} = \frac{k_r}{m_r} - 3 + \Delta, \qquad \Theta_{r} = \frac{k_r}{m_r} - \Delta,
	\end{equation*}
	and since $0\leq 2k_r\leq m_r$ and $3>\Delta\geq\frac{3+\sqrt{5}}{2}$, we have $-3 < \Theta_r < -2 < 0 < \theta_r < 1$. Moreover note that $\theta_r=\overline{\Theta_r}$ where the bar denotes the Galois conjugate.

	The continued fraction expansion of such roots can be described through Christoffel words. Frobenius associated with each Markoff number $m > 2$ an ordered pair of relatively prime positive integers, which are called \emph{Frobenius coordinates}. According to \cite[Page 24]{CF89}, given a pair of coprime positive integers $(\mu,\nu)$ with $\nu>1$ with $m_r$ the corresponding Markov number, let
	\begin{equation}\label{eq:Christoffel_word}
		r_i = \left\lfloor i\frac{\mu}{\nu}\right\rfloor - \left\lfloor (i-1)\frac{\mu}{\nu}\right\rfloor, \quad 1\leq i\leq \nu.
	\end{equation}
	The word $w(\mu,\nu)=ab^{r_1}ab^{r_2}\dots ab^{r_\nu}$ is precisely the lower Christoffel word such that $|w(\mu,\nu)|_b=\mu$ and $|w(\mu,\nu)|_a=\nu$ (see \cite[Chapter 2]{Reutenauerbook}). In particular there is an alphabet $(\alpha_r,\beta_r)\in\overline{P}$ such that $ab^{r_1}ab^{r_2}\dots ab^{r_\nu}=\alpha_r\beta_r$.
	
	By \cite[Page 27]{CF89}, the continued fraction of $\theta_r$ is
	\begin{equation}\label{eq:theta_r}
		\theta_r = [0;2,\overline{b^{r_1},a,b^{r_2},a,\dots,a,b^{r_\nu},a}],
	\end{equation}
	where we used the substitution $a=(2,2)$ and $b=(1,1)$. If $\nu=1$ then we have $\theta_r=[0;2,\overline{b^{\mu},2,2}]$. In particular we have $2+\theta_r=\gamma_r$ where $\gamma_r$ was defined on \eqref{eq:gamma_r}. An elementary identity of continued fractions useful for our purposes is 
	\begin{equation}\label{eq:elementary_identity}
		[0;1+x,y]=1-[0;1,x,y].
	\end{equation}
	
	The following is \cite[Lemma 2.1]{Florek} with a different proof (see also \cite[Lemma 19]{Bombieri}).
	
	\begin{lemma}\label{lem:lemma21_Florek}
		If $\nu>1$, the continued fraction of $\Theta_r$ is
		\begin{equation*}
			\Theta_r + 3 = [0;2,\overline{b^{r_1},a,b^{r_2},a,\dots,a,b^{r_{\nu}-1},a,b}].
		\end{equation*}    
		If $\nu=1$ then the continued fraction is $\Theta_r=[0;2,\overline{b^{\mu-1},a,b}]$.
	\end{lemma}
	
	\begin{proof}
		It follows from \cite[Page 23]{CF89} that the word $b^{r_1}\dots ab^{r_\nu-1}$ is a palindrome. Since $\theta_r+2$ is a purely periodic continued fraction with period $ab^{r_1}\dots ab^{r_\nu}$, we have that the Galois conjugate has continued fraction 
		\begin{equation*}
			-\Theta_r-2=-\overline{\theta_r}-2=[0;\overline{b^{r_\nu},\dots,b^{r_1},a}]=[0;\overline{b^{r_1+1},\dots a,b^{r_\nu-1},a}]
		\end{equation*}
		In particular using \eqref{eq:elementary_identity} we conclude
		\begin{equation*}
			\Theta_r+3=1-[0;1,1,\overline{b^{r_1},\dots a,b^{r_\nu-1},a,b}]=[0;2,\overline{b^{r_1},a,\dots,a,b^{r_{\nu}-1},a,b}].
		\end{equation*}
	\end{proof}
	
	Since $\alpha_r\beta_r$ is a lower Christoffel word that it is not a letter, it is of the form $\alpha_r\beta_r=a\theta^{(r)} b$ with $\theta^{(r)}$ a palindrome. By the previous lemma we have 
	\begin{equation}\label{eq:Theta_r_plus_3}
		\Theta_r+3=[0;2,\overline{\theta^{(r)},a,b}].
	\end{equation}

	\subsubsection{The spectrum \texorpdfstring{$\widetilde{M}$}{Tilde-M} before 3}\label{subsec:spectrum_widetilde_M_before_3}
	
	Now we will explain the link between \Cref{thm:characterization_tildeM}, \Cref{thm:theorem_Florek_Harcos} and the case $N=0$ of \Cref{thm:full_characterization}. 
	Notice that by \Cref{prop_renorm}, the irrationals in the non-periodic and $N=0$ case of \Cref{thm:full_characterization} have Lagrange value $k=3$, in particular, must have value $\tilde{m}=3$.
	So to characterize the spectrum $\widetilde{M}$ before 3, it suffices to consider the periodic and $N=0$ case of \Cref{thm:full_characterization}.
	
	All continued fractions $x=[0;x_1,x_2,\dots]$ that are eventually periodic with period $uv$ where $(u,v)\in\overline{P}$ and such that $\widetilde{m}(x)<3$, can be found by \Cref{thm:full_characterization}, by choosing pairs $(\alpha,\beta)\in\overline{P}$ such that $\alpha=uv$ or $\beta=uv$. The only continued fractions that are not of this form come from pairs $(\alpha,\beta)=(a,a^nb)$ or $(\alpha,\beta)=(ab^n,b)$, explicitly they are
	\begin{equation*}
		x_1x_2\dots=b^\infty=111\dots, \quad x_1x_2\dots=2b^\infty=2111\dots
	\end{equation*}
	with value $\widetilde{m}(x)=\frac{3+\sqrt{5}}{2}$ and
	\begin{equation*}
		x_1x_2\dots=a^\infty=222\dots, \quad x_1x_2\dots=ba^\infty=11222\dots
	\end{equation*}
	with value $\widetilde{m}(x)=\frac{3+\sqrt{8}}{2}$. For the rest of the numbers it suffices to consider
	\begin{equation}\label{eq:expansion-cases}
		\begin{aligned}
			x_1x_2\cdots &= ((uv)^T)^\infty,
			&\quad\text{where } (\alpha,\beta)=(u,uv),\\
			x_1x_2\cdots &= v^T((uv)^T)^\infty,
			&\quad\text{where } (\alpha,\beta)=(uv,v),\\
			x_1x_2\cdots &= 2u^+(uv)^\infty,
			&\quad\text{where } (\alpha,\beta)=(uv,v),\\
			x_1x_2\cdots &= 2u^+v(uv)^\infty,
			&\quad\text{where } (\alpha,\beta)=(u,uv).
		\end{aligned}
	\end{equation}
	
	Since $(uv)^T=(v_au^b)^T=u^bv_a$ (because both are palindromes), we have
	\begin{equation}\label{eq:periodic_identity1}
		((uv)^T)^\infty=(u^bv_a)^\infty=u^b(v_au^b)^\infty=u^b(uv)^\infty
	\end{equation}
	and
	\begin{equation}\label{eq:periodic_identity2}
		v^T((uv)^T)^\infty=v^T(u^bv_a)^\infty=v^Tu^b(uv)^\infty=u^bv(uv)^\infty.
	\end{equation}
	
	In particular, the continued fraction $[0;2,u^{+},v,\overline{u,v}]$ is equal to a positive root $\theta_r$ of some Markov form. Write $uv=a\theta^{(r)}b$ for some palindrome $\theta^{(r)}$. By \eqref{eq:elementary_identity} and \eqref{eq:Theta_r_plus_3} we have that 
	\begin{align*}
		[0;\overline{(uv)^T}]&=[0;\overline{b,\theta^{(r)},a}]=1-[0;2,\overline{\theta^{(r)},a,b}]=-2-\Theta_r, \\
		[0;v^T,\overline{(uv)^T}]&=[0;u^b,v,\overline{u,v}]=1-[0;2,u^{+},v,\overline{u,v}]=1-\theta_r, \\
		[0;2,u^{+},\overline{u,v}]&=1-[0;u^b,\overline{u,v}]=1-[0;\overline{(uv)^T}]=3+\Theta_r.
	\end{align*}
	
	Finally, using the fact that $\widetilde{m}(x)=\widetilde{m}(x+k)$ for any irrational $x$ and $k\in\Z$, one gets that $\widetilde{m}(k\pm\theta_r)<3$ and $\widetilde{m}(k\pm\Theta_r)<3$ for any roots $\theta_r,\Theta_r$ and any $k\in\Z$ and furthermore these are all possibilities.

	\subsection{Cuts}
	
	We use a vertical bar between two letters to represent a \emph{cut} of words over the alphabet $\{1,2\}$. 
	For example, if $x=x_1x_2x_3\dots$ is an infinite word, the cut at the $i$-th position is $x_1\dots x_{i-2}x_{i-1}\vert x_{i}\dots$. 
	We define the \emph{value} of a cut of infinite or bi-infinite words by
	\begin{align*}
		\lambda(x_1\dots x_{i-1}x_i\vert x_{i+1}\dots)&=[0;x_i,x_{i-1},\dots,x_1]+[x_{i+1};x_{i+2},\dots] \\
		\lambda(\dots x_{i-1}x_i\vert x_{i+1}\dots)&=[0;x_i,x_{i-1},\dots]+[x_{i+1};x_{i+2},\dots]
	\end{align*}
	where we use the same notation as the height function $\lambda$, which is an abuse of notation.
	
	When considering words over the alphabet $\{a,b\}$, generally we have two types of cuts: $E^Tb\vert aF$ or $E^T a\vert bF$, where $E,F$ are two infinite words in alphabet $\{a,b\}$.
	The first one is clear, while for the second one we adopt the following abuse of notation:
	\begin{equation*}
		\lambda(E^Ta|b F) = [2;2,E]+[0;1,1,F] = \lambda(E^T2|2b F).
	\end{equation*}
	So when writing $E^Ta\vert bF$, we always mean the cut $E^T2\vert 2bF$ instead of the cut $E^T22\vert bF$.
	Since almost all the words appearing in this paper can be written over the alphabet $\{a,b\}$, this notation is particularly convenient for applying the following lexicographic comparison criteria (\Cref{prop_lexicographic_comparison}).
	
	\begin{definition}\label{def_badcut}
		A cut of an infinite or bi-infinite word is \emph{bad} if the cut has value strictly bigger than $3$, otherwise the cut is \emph{good}.
		
		We can extend the definition to finite words. 
		The cut of finite word is \emph{bad} (resp. good) if any extension of the cut over the alphabet $\{a,b\}$ is still a bad cut (resp. good cut).
		Otherwise we say that the cut is \emph{indeterminate}.
	\end{definition}

	The following proposition is from \cite[Lemma 8]{Bombieri}, which gives a criteria to deduce if a cut is good or bad.
	\begin{proposition}\label{prop_lexicographic_comparison}
		Let $\preccurlyeq$ be the lexicographic order on words of alphabet $\{a,b\}$ and let $E,F$ be two infinite words over the alphabet $\{a,b\}$.
		Then the cuts $E^Tb\vert aF$ and $F^Ta\vert bE$ are good if and only if $E\preccurlyeq F$, and are bad if and only if $F\preccurlyeq E$.
		
		As a consequence, the cuts $a\omega^Tb\vert a\omega b $, $b\omega^Ta\vert b\omega a$ are good and the cuts $b\omega^Tb\vert a\omega a$, $a\omega^Ta\vert b\omega b$ are bad, where $\omega$ is a finite word in alphabet $\{a,b\}$.
	\end{proposition}

	Given finite word $\omega=\omega_1\omega_2\dots \omega_n$ of alphabet $\N_{>0}$, define \emph{the size} of $\omega$ by $\sizes(\omega)=\left|I(\omega)\right|$, where $I(\omega)$ is a closed subinterval of $[0,1]$ given by
	\begin{equation*}
		I(\omega)=\left\{ x\in[0,1] \mid x=[0;\omega_1,\omega_2,\dots,\omega_n,t], t\geq 1 \right\}\cup\left\{ [0;\omega_1,\omega_2,\dots,\omega_n] \right\}.
	\end{equation*}
	That is, $I(\omega)$ consists of all numbers in $[0,1]$ whose continued fraction expansions begin with $\omega$.
	The following lemma is a basic property of function $\lambda$.
	
	\begin{lemma}\label{lem:compare}		
		Let $\gamma,\eta$ be two finite words over the alphabet $\{1,2\}$ of even length such that $\gamma$ is a strict prefix of $\eta$, and let $E^Tb\vert aF=\dots a\gamma^Tb\vert a\gamma b\dots, (E^\prime)^Tb\vert aF^\prime=\dots a\eta^Tb\vert a\eta b\dots$ be two bi-infinite words over the alphabet $\{1,2\}$.
		We have
		\begin{equation*}
			\lambda(E^Tb\vert aF)<\lambda((E^\prime)^Tb\vert aF^\prime)
		\end{equation*}
	\end{lemma}
	
	\begin{proof}
		By \cite[Lemma 3.2]{EGRS2024} we have
		\begin{equation*}
			\lambda(E^Tb\vert aF) < 3 - \sizes(b\gamma b) \leq 3-\sizes(b\eta) < 3-\sizes(b\eta 1) < \lambda((E^\prime)^Tb\vert aF^\prime).
		\end{equation*}		
	\end{proof}

	\subsection{Balanced sequences}\label{subsec:balanced}
	
	The balanced bi-infinite words were classified by Heinis \cite[Theorem 2.5]{Heinis}. The following is a characterization of balanced sequences based on Christoffel words. We will not use it, but is instructive to understand the type of sequences that appear in  \Cref{thm:full_characterization}.  
	
	\begin{proposition}\label{prop:balanced_characterization}
		A bi-infinite word $\omega\in\{a,b\}^{\Z}$ is a balanced sequence if and only if it has the following forms
		\begin{itemize}
			\item Periodic: $\omega=\theta^{\infty}$ for some Christoffel word $\theta\in P\cup\{a,b\}$.
			\item Degenerate: $\omega=\alpha^{\infty}\beta\alpha^{\infty}$ or $\omega=\beta^{\infty}\alpha\beta^{\infty}$ for some $(\alpha,\beta)\in\overline{P}$.
			\item Not eventually periodic (Sturmian sequences): $\omega$ if for all $n\in\N$, the word $\omega$ can be written as a non-constant word for some alphabet $(\alpha_n,\beta_n)\in\overline{P}_n$. 
		\end{itemize}
	\end{proposition} 
	
	In \cite{reutenauer}, it was proven that the bi-infinite sequences with Markov value less or equal than 3 are the same as the bi-infinite balanced sequences after replacing $a=(2,2)$ and $b=(1,1)$. One can further conclude that such a bi-infinite balanced sequence has Lagrange value equal to 3 if and only if it is not periodic.  
	
	\subsubsection{Lower and upper balanced sequences}
	
	The lower and upper balanced sequences were characterized by Gurwood \cite{Gurwood}. We will reproduce this of his thesis with slightly changes for completeness.
	
	\begin{definition}
		Let $\omega=w_1w_2\dots$ be an right infinite word over $\{a,b\}$.  We say that $\omega$ is \emph{upper balanced} if $\omega=b^\infty$ or if it is not constant and
		\begin{enumerate}
			\item any cut of the form $\dots b|a\dots$ extends to a good cut or extends to $\theta^T b|a\theta$ for some prefix $\theta^Tba\theta$ of $\omega$,
			\item any cut of the form $\dots a|b\dots$ extends to a good cut.
		\end{enumerate}
		We say that $\omega$ is \emph{lower balanced} if $\omega=a^\infty$ or it is not constant and
		\begin{enumerate}
			\item any cut of the form $\dots a|b\dots$ extends to a good cut or extends to $\theta^T a|b\theta$ for some prefix $\theta^Tab\theta$ of $\omega$,
			\item any cut of the form $\dots b|a\dots$ extends to a good cut.
		\end{enumerate}
		We say that a left infinite word $\omega=\dots w_{2}w_{1}$ is lower (resp. upper) balanced if its transpose $\omega^T=w_1w_2\dots$ is upper (resp. lower) balanced.
	\end{definition}
	
	It follows directly from the definition that lower (resp. upper) balanced sequences begin with $a$ (resp. $b$). It also follows that a right infinite sequence $\omega$ is lower balanced if and only if $b\omega^{+}$ is upper balanced.
	
	The following three lemmas are from Gurwood \cite{Gurwood}. We reproduce their proof for completeness. 
	
	\begin{lemma}\label{lem:lemma31gurwood}
		Suppose that $w_1\dots w_n\in\{a,b\}^n$ is not balanced. Then any extension $w_1\dots w_{2n}\in\{a,b\}^{2n}$ contains a bad cut, that is, contains $x\theta^T x|y\theta y$ with $\{x,y\}=\{a,b\}$.
	\end{lemma}
	\begin{proof}
		The proof is by induction on $n$. The case $n\leq 3$ can be checked directly. Assume that the theorem is true for all $m\leq n-1$. By contradiction, assume that there is an extension  $w_1\dots w_{2n}$ with no bad cuts. By the induction hypothesis, we can assume that $\big||w_1\dots w_{k}|_b-|w_{n-k+1}\dots w_{n}|_b\big|>1$ and that $k$ is minimal with that property. Since $k$ is minimal, we can assume $|w_{n-k+1}\dots w_{n}|_b-|w_1\dots w_{k}|_b=2$ (if the value is $-2$ we can exchange $a$ and $b$). Let $M=|w_1\dots w_{k}|_b+1=|w_{n-k+1}\dots w_{n}|_b-1$. We can also assume $n$ is minimal in the sense that any strict factor of $w_1\dots w_n$ is balanced, so we have $|w_{j+1}\dots w_{j+k}|_{b}=M$ for $1\leq j\leq n-k-1$. In particular, we have that $w_s=w_t \pmod{k}$, $2\leq s,t\leq n-1$. Moreover, $w_1=w_k=w_{n-k}=a$ and $w_{n-k+1}=w_n=w_{k+1}=b$. Set $A_\ell=|w_{n-\ell+1}\dots w_n|_b-|w_1\dots w_\ell|_b$. In particular we have $A_1=1$, $A_k=2$ and by the minimality of $k$ that $A_\ell\leq 1$ for $1\leq \ell\leq k-1$. If $A_\ell=0$ for some $2\leq\ell\leq k-1$, then $|w_{n-k+1}\dots w_{n-\ell}|_b-|w_\ell\dots w_k|_b=2$, which contradicts the induction hypothesis. Hence $A_\ell=1$ for $1\leq\ell\leq k-1$ and thus $w_s=w_{n-s+1}$ for $2\leq s\leq k-1$. 
		
		Let $n=qk+r$, $0\leq r\leq k-1$. If $2\leq r\leq k-1$ then $w_r=w_{(q-1)k+r}=w_{n-k}=a$ and $w_{n-r+1}=w_{(n-r+1)-(q-1)k}=w_{k+1}=b$ which is a contradiction and if $1\leq r\leq k-2$ then $w_{r+1}=w_{(q-1)k+r+1}=w_{n-k-1}=b$ and $w_{n-r}=w_{(n-r)-(q-1)k}=w_k=a$ which is also a contradiction. Therefore $r=0$ and the word $w_1\dotsb w_n$ has the form
		\begin{equation*}
			w_1\dotsb w_n = a \theta a~b\theta a~b\theta a~\dotsb~b\theta a~b\theta b
		\end{equation*}
		where the word $\theta = w_2\dots w_{k-1}$ is a palindrome. 
		
		Now, recall that we assume by contradiction that $w_1\dots w_{2n}$ does not have bad cuts. In particular, by considering the cut 
		\begin{equation*}
			\dots w_{n-k}|w_{n-k+1}\dots = \dotsb a~b\theta a|b\theta b \dotsb 
		\end{equation*}
		we must have $w_{n+1}=a$. Now we want to use the cut $w_n|w_{n+1}= b|a$, so lets consider the maximal $2\leq t\leq n-2k-1$ such that $w_{n+2}\dotsb w_{n+t} = w_{n-1}\dotsb w_{n-t+1}$. Since there are no bad cuts, if $w_{n-t-k}=w_{n-t}=b$ then $w_{n+t+1}=b$. Since $w_{n-t+1}\dotsb w_{n-1}=w_{n-t-k+1}\dotsb w_{n-k-1}$ ,by considering the cut
		\begin{equation*}
			\dots w_{n-k}|w_{n-k+1}\dots = \dotsb w_{n-t-k}~w_{n-t-k+1}\dotsb w_{n-k-1}a~b\theta a|b\theta b~aw_{n+2}\dotsb w_{n+t}~w_{n+t+1} \dotsb,
		\end{equation*}
		we have that if $w_{n-t-k}=a$, then $w_{n+t+1}=a$. In resume, we must have that $w_{n+t+1}=w_{n-t-k}=w_{n-t}$. Since $t$ is maximal, we must have that $t=n-2k-1$, which implies $a=w_{1}=w_{n-t-2k}=w_{n-t-k}=w_{n-t}=w_{2k+1}=b$, a contradiction.  
	\end{proof}
	
	\begin{lemma}\label{lem:lemma32gurwood}
		The non-constant right infinite word $\omega=w_1w_2\dots$ over the alphabet $\{a,b\}$ is lower balanced if and only if
		\begin{equation}\label{eq:lower_balanced_equation}
			|w_1\dots w_k|_b\leq |w_{\ell+1}\dots w_{\ell+k}|_b\leq |w_1\dots w_k|_b+1
		\end{equation}
		for all $\ell,k\geq 1$.
	\end{lemma}
	\begin{proof}
		By \Cref{lem:lemma31gurwood} we must have that the whole sequence $\omega$ is balanced. In particular, we have that $\big||w_1\dots w_k|_b- |w_{\ell+1}\dots w_{\ell+k}|_b\big|\leq 1$. Since the sequence $b\omega^{+}$ is upper balanced, it does not have bad cuts either, so again by \Cref{lem:lemma31gurwood} is balanced and $\big|1+|w_1\dots w_k|_b- |w_{\ell+1}\dots w_{\ell+k}|_b\big|\leq 1$. These two equations imply \eqref{eq:lower_balanced_equation}. 
		
		Now suppose \eqref{eq:lower_balanced_equation}. Notice that a bad cut $x\theta^Tx|y\theta y$, $\{x,y\}=\{a,b\}$ is incompatible with \eqref{eq:lower_balanced_equation} because $\big||x\theta^Tx|_b-|y\theta y|_b\big|>1$. Finally, if a cut of $\omega$ of the form $\dots b|a\dots$ extends to $\theta^Tb|a\theta$ for some prefix of $\omega$, then $|\theta^Tb|_b>|a\theta|_b$ which contradicts \eqref{eq:lower_balanced_equation}. In conclusion, any such cut $\dots b|a\dots$ must extend to a good cut.
	\end{proof}
	
	Let $\chi$ be the substitution given by $\chi(a)=0$ and $\chi(b)=1$. 
	\begin{lemma}\label{lem:balanced_inequality}
		The right infinite word $\omega=w_1w_2\dots\neq b^\infty$ over the alphabet $\{a,b\}$ is lower balanced if and only if there is $0\leq\xi\leq1$ such that either
		\begin{equation}\label{eq:irrational_xi}
			\chi(w_n)=\lfloor n\xi \rfloor -\lfloor (n-1)\xi \rfloor, \qquad \text{for all }n\geq 1,
		\end{equation}
		or $w_1=a$ and
		\begin{equation}\label{eq:skew_xi}
			\chi(w_n)=\lfloor -(n-1)\xi \rfloor - \lfloor -n\xi \rfloor, \qquad \text{for all }n\geq 2.
		\end{equation}
	\end{lemma}
	
	\begin{proof}
		The fact that these sequences are lower balanced follows from the inequality $\lfloor x\rfloor +\lfloor y\rfloor\leq \lfloor x+y\rfloor\leq \lfloor x\rfloor +\lfloor y\rfloor+1$ for arbitrary real numbers $x,y$.
		
		Conversely, let $S_n=|w_1\dots w_n|_b$. By \Cref{lem:lemma32gurwood} we have $S_n+S_m\leq S_{n+m}\leq S_n+S_m+1$ for all $m,n$, that is, the sequence $(S_n)_{n\geq 1}$ is almost sub-additive. By induction, we have that $mS_n\leq S_{mn}\leq nS_m+n-1$ and thus $\frac{S_n}{n}<\frac{S_m+1}{m}$. 
		Set $\xi=\sup_{n\geq 1}\frac{S_n}{n}\in[0,1]$, we have $\frac{S_m}{m}\leq\xi\leq\frac{S_m+1}{m}$. If $\xi$ is irrational then $\frac{S_m}{m}<\xi<\frac{S_m+1}{m}$ for all $m$, so $S_m=\lfloor m\xi\rfloor$ for all $m\geq 1$. 
		If $\xi$ is rational and $\xi=\frac{S_k}{k}$ for some $k$, then $\frac{S_m}{m}\leq\xi<\frac{S_m+1}{m}$ for all $m$, so $S_m=\lfloor m\xi\rfloor$ for all $m\geq 1$. In both cases $\chi(w_n)=S_{n+1}-S_n=\lfloor n\xi \rfloor -\lfloor (n-1)\xi \rfloor$ for all $n\geq 1$. If $\xi\neq \frac{S_k}{k}$ for all $k$, then $\xi\neq 0$ and $\frac{S_m}{m}<\xi\leq\frac{S_m+1}{m}$ so $m\xi\leq S_m+1<m\xi+1$. Hence $S_m+1=-\lfloor -m\xi\rfloor$, which implies that $S_1=0$ and $\chi(w_n)=\lfloor -(n-1)\xi \rfloor - \lfloor -n\xi \rfloor$ for all $n\geq 2$.
	\end{proof}

	\subsubsection{Characterization of the sequence \texorpdfstring{$\lfloor n\xi \rfloor - \lfloor (n-1)\xi \rfloor$}{floor(n xi) - floor((n-1) xi)} through renormalization}\label{charcterization_xi}
	
	For $0<\xi<1$, let $\omega=w_1w_2\dots$ a right infinite word over the alphabet $\{a,b\}$ such that $\chi(w_n)=\lfloor n\xi \rfloor -\lfloor (n-1)\xi \rfloor$.
	If $\xi=\frac{p}{q}$ is rational, then $\omega$ is periodic with period $\eta=w_1\dots w_q$. Moreover, we have $w(p,q)=U(\eta)$, so $\eta$ is exactly the Christoffel lower word with Frobenius coordinates $|\eta|_a=q-p$ and $|\eta|_b=p$.
	If $\xi$ is irrational, then by \cite[Page 288]{ItoYasutomi} the sequence $\omega$ can be described by the renormalization operators $U,V$. Indeed, if $\xi=[0;d_1,d_2,\dots]$ is the continued fraction of $\xi$ then letting $i_n=d_n$ and $R_n=U$ if $n$ is even and $i_n=d_n-1$ and $R_n=V$ is $n$ is odd, we have that $w_1w_2w_3\dots=\lim_{n\to\infty}(R_1^{i_1}\dotsb R_n^{i_n})(a)$, so by \Cref{relbtrenorm} we have $\omega=\lim_{n\to\infty}\alpha_n$  where   $(\alpha_n,\beta_n)=\left(\overline{R_n}^{i_n}\dotsb\overline{R_1}^{i_1}\right)(a,b)$.

	\subsubsection{Indeterminate cuts in one--sided balanced sequences}
	
	In this subsection we will determine how are the indeterminate cuts of certain balanced sequences. This information is what allows to go from the case of no solutions of \eqref{eq:x-pq<3} to the case of any finite number of solutions.
	
	\begin{lemma}\label{lem:ba_good_cuts}
		Let $w=\alpha\beta$ for some $(\alpha,\beta)\in\overline{P}$. Let us write $w=w_1\dots w_k$ with $w_i\in\{a,b\}$ for $1\leq i\leq k$. Then for any $1\leq j\leq k$ it holds that: if $w_jw_{j+1}=ba$, then there is $r\geq 1$ such that $\theta=w_{j-1}\dots w_{j-r+1}=w_{j+2}\dots w_{j+r}$, $w_{j-r}=a$ and $w_{j+r+1}=b$, so in particular $w_{j-r}\dots w_{j+r+1}=a\theta^T ba\theta b$. 
	\end{lemma}
	
	\begin{proof}
		Recall that given any such $w$, there is a sequence of renormalization operators $R_1,\dots,R_n\in\{U,V\}$ such that $w=(R_n\dotsb R_1)(ab)$. The proof of the lemma is by induction on $n$. 
		
		Since the base case $w=ab$ is trivial, suppose that $\tilde{w}=U(w)$. In this case, any factor $ba$ of $\tilde{w}$ comes from either $U(a|a)=ab|ab$ or $U(b|a)=b|ab$. In the second case, by hypothesis $a\theta^Tb|a\theta b$ is a subword of $w$, so $U(a\theta^Tb|a\theta b)=abU(\theta^T)b|abU(\theta)b=aU(\theta)^Tbb|abU(\theta)b=a\tilde{\theta}^Tb|a\tilde{\theta}b$ where $\tilde{\theta}=bU(\theta)$. 
		In case that $\tilde{w}=V(w)$, then any factor $ba$ of $\tilde{w}$ comes from either $V(b|b)=ab|ab$ or $V(b|a)=ab|a$. In the second case, by hypothesis $a\theta^Tb|a\theta b$ is a subword of $w$, so $V(a\theta^Tb|a\theta b)=aV(\theta^T)ab|aV(\theta)ab=aaV(\theta)^Tb|aV(\theta)ab=a\tilde{\theta}^Tb|a\tilde{\theta}b$ where $\tilde{\theta}=V(\theta)a$.
	\end{proof}
	
	The conclusion of the following lemma is that the right infinite words $\lim_{n\to\infty}\beta_n^T$ and $((\alpha\beta)^T)^\infty$ are upper balanced while $\lim_{n\to\infty}\alpha_n$ and $(\alpha\beta)^\infty$ are lower balanced.
	
	\begin{lemma}\label{lem:safe_cuts}
		Let $\overline{R}_1,\overline{R}_2,\dots$ be a sequence of renormalization operators $\overline{U},\overline{V}$ such that both the operators $\overline{U},\overline{V}$ appears infinitely many times and let $(\alpha_0,\beta_0)=(a,b)$ and $(\alpha_{n+1},\beta_{n+1})=\overline{R}_{n+1}(\alpha_{n},\beta_{n})\in\overline{P}_{n+1}$ for $n\geq 0$.
		All the cuts of the left infinite words  $\lim_{n\to\infty}\beta_n$, $(\alpha\beta)^\infty$ and the cuts of the right infinite words $\lim_{n\to\infty}\alpha_n$, $(\alpha\beta)^\infty$ are either good or indeterminate. Moreover, if they are indeterminate, then they are of the form $\theta^Ta|b\theta$, with $\theta$ of even length, for some suffix or prefix, respectively, or they are of form $2|2a\dots$ in a right infinite word.
	\end{lemma}
	\begin{proof}
		First notice that the possible patterns of indeterminate cut is of form $2|2a\dots$, $\dots b|a\dots$, or $\dots a|b\dots$.
		The first cut only occur in a right infinite word beginning with $aa$, and the second is ruled out by \Cref{lem:ba_good_cuts}.
		
		For the cuts $\dots a|b\dots$, since it can not extend to $a\theta^T a|b\theta b$ because otherwise that subword would be contained in some $\beta_n$ (which contradicts the fact that the bi-infinite periodic sequence $(\beta_n)^\infty$ has no bad cuts), we have that the cut $\dots a|b\dots$ either it extends to $b\theta^T a|b\theta a$ (which is good) or it extends to $\theta^T a|b\theta$ being also a suffix of $\lim_{n\to\infty}\beta_n$ or prefix $\lim_{n\to\infty}\alpha_n$.
	\end{proof}

	\begin{lemma}\label{lem:extend_alphabeta}
		Let $(\alpha,\beta)\in\overline{P}_n$ for some $n\geq 1$.
		\begin{enumerate}
			\item\label{item_alphabetaright} If $(\alpha,\beta)\neq (ab^{n},b)$, then there are finite words $\eta,\theta,\theta^\prime$ such that $\eta\neq\varnothing$, $\theta$ is a prefix of $\theta^\prime$, and that $\alpha\beta$ has a factorization $\eta\theta^T a|b\theta$.
			Moreover, the exactly same cut can be extended to some subword $(\theta^\prime)^Ta|b\theta^\prime$ inside $\alpha\beta\alpha$.  
			\item\label{item_alphabetaleft} If $(\alpha,\beta)\neq (a,a^{n}b)$, then there are finite words $\eta,\theta,\theta^\prime$ such that $\eta\neq\varnothing$, $\theta$ is a prefix of $\theta^\prime$, and that $\alpha\beta$ has a factorization $\theta^T a|b\theta\eta$, 
			Moreover, the exactly same cut can be extended to some subword $(\theta^\prime)^Ta|b\theta^\prime$ inside $\beta\alpha\beta$.  
		\end{enumerate}
		Moreover, if $(\alpha,\beta)$ is obtained thorough a sequence of renormalization operators $R_1,R_2,\dots$, $R_i\in\{U,V\}$, such that both $U$ and $V$ appear infinitely many times, then $|\theta^\prime|-|\theta|$ goes to infinity as $n\to\infty$.
	\end{lemma}
	
	\begin{proof}
		
		We only prove \Cref{item_alphabetaright} since the another is analogous.
		
		First notice that by \Cref{relbtrenorm}, if $(\alpha_n,\beta_n)\neq (ab^{n},b)$ then there exist some $0\leq k<n$ so that we can write $(\alpha_1,\beta_1)=(a,ab)$ or $(\alpha_n,\beta_n)=(R_n\dots R_{k+2})(a(ab)^{k},ab)$ where $R_i\in\{U,V\}$.
		We use induction on $n>k$.
		First for the base cases $(a(ab)^k,ab)$, it is enough to take $\eta=a(ab)^k,\theta=\theta^\prime=\varnothing$.
		Now suppose the claim holds for $\alpha_n\beta_n=\eta_n\theta_n^T ab\theta_n$ for some $n\geq k+1$, we prove for the case $\alpha_{n+1}\beta_{n+1}=R_{n+1}(\alpha_n\beta_n)$.
		
		If $R_{n+1}=U$, since $\eta_n\neq\varnothing$, we let $\eta_{n+1}b=U(\eta_n)$ and $\theta_{n+1}=bU(\theta_n)$. Therefore
		\begin{equation*}
			\alpha_{n+1}\beta_{n+1}=\eta_{n+1}bU(\theta_n^T)a|bbU(\theta_n)=\eta_{n+1}U(\theta_n)^Tba|bbU(\theta_n)=\eta_{n+1}\theta_{n+1}^Ta|b\theta_{n+1},
		\end{equation*}
		and then take $\theta_{n+1}^\prime=bU(\theta_{n}^\prime)$ if $(\theta_{n}^\prime)^T$ is not a prefix of $\alpha_{n}$ and $\theta_{n+1}^\prime=b(U(\theta_{n}^\prime))^{-}$ otherwise.    	
		Notice that if $(\theta_{n}^\prime)^T$ is prefix of $\alpha_{n}$, then $\theta_{n}^\prime$ ends with $a$ and since $\eta_n\neq\varnothing$ we have that $\theta_{n}$ is a strict prefix of $\theta_{n}^\prime$, so indeed $bU(\theta_{n})$ is still a strict prefix of $b(U(\theta_{n}^\prime))^{-}$. 
		
		If $R_{n+1}=V$, let $\eta_{n+1}=V(\eta_n)a$, $\theta_{n+1}=V(\theta_n)$. We have
		\begin{equation*}
			\alpha_{n+1}\beta_{n+1}=V(\eta_n)V(\theta_n^T)aa|bV(\theta_n)=V(\eta_{n})aV(\theta_n)^Ta|bV(\theta_n)=\eta_{n+1}\theta_{n+1}^Ta|b\theta_{n+1},
		\end{equation*}
		and then take $\theta_{m+1}^\prime=V(\theta_{n}^\prime)a$ if $\theta_{n}^\prime$ is not suffix of $\alpha_{n}$ or $\theta_{n+1}^\prime=V(\theta_{n}^\prime)$ otherwise.
		
		Finally, notice that $|\theta_{n+1}^\prime|-|\theta_{n+1}|$ will increase each time we apply $V$ and $\theta_{n}^\prime$ is not a suffix of $\alpha_{n}$. Otherwise if $\theta_{n}^\prime$ is a suffix of $\alpha_n$, then we actually have $\theta_{n}^\prime=\theta_{n}\alpha_{n}$ so $|\theta_{n}^\prime|-|\theta_{n}|$ is also growing.
	\end{proof}
	
	\begin{remark}
		The assumption on $(\alpha,\beta)$ is to guarantee the existence of such factorization.
		For example, if $(\alpha,\beta)=(ab^n,b)$ then $\alpha\beta=ab^{n+1}$ can not be written in form $\eta\theta^Ta|b\theta$.
	\end{remark}
	
	\begin{lemma}\label{lem:only_one_indeterminate_cut}
		Let $(\alpha,\beta)\in\overline{P}_n$. 
		\begin{enumerate}
			\item The left infinite words $\alpha^\infty\beta$ and $(\alpha\beta)^\infty$ contain exactly one indeterminate cut $R\theta^Ta|b\theta$. This cut is located inside the suffix $\alpha\beta$. \label{item:left_infinite}
			\item The right infinite words $\alpha\beta^\infty$ and $(\alpha\beta)^\infty$ contain exactly one indeterminate cut $\theta^Ta|b\theta R$ if $ab$ is a prefix of $\alpha\beta$, and exactly the two indeterminate cuts $\theta^Ta|b\theta R$, $2|2a\dots$ if $aa$ is a prefix of $\alpha\beta$. All these cuts are located inside the prefix $\alpha\beta$. \label{item:right_infinite}
		\end{enumerate}
	\end{lemma}
	\begin{proof}
		Notice that by \Cref{lem:safe_cuts} all cuts of such infinite words are either good or indeterminate, and the latter only occurs when they are suffixes and prefixes, accordingly. Moreover, those indeterminate cuts are all of the form $\theta^Ta|b\theta$, except for the right infinite word beginning with $aa$.
		We only prove \Cref{item:left_infinite} since the another case is analogous.
		
		We can assume that $(\alpha,\beta)=(\alpha_n,\beta_n)\in\overline{P}_n$ for some $n\geq 1$, so by \Cref{relbtrenorm} there are renormalization operators $R_1,\dots,R_n\in\{U,V\}$ such that $(\alpha_n,\beta_n)=(R_n\dots R_1)(a,b)$. 
		We will prove the claim by induction on $n$. Clearly the statement hold for $(ab)^\infty$ and $a^\infty b$. Consider first the situation when $R_{n+1}=U$. Notice that the new cuts $a|b$  come from $U(a)=ab$. Since any word ends with $b$ after applying $U$ and since $a|bU(a)=a|bab$, we see that any indeterminate cut in the new sequence comes from $U(ab)=a|bb$. 
		If the $ab$ is not in the last $\alpha_n\beta_n$, then by the induction hypothesis extends to a good cut $b\theta^Ta|b\theta a$, so it becomes $bU(b\theta^Ta|b\theta a)=bbU(\theta^T)a|bbU(\theta)ab=bU(\theta)^Tba|bbU(\theta)ab$ which is good. 
		In the case that this $ab$ is located in the last $\alpha_n\beta_n$ and is indeterminate, it produces the only new indeterminate cut  $bU(\theta^T)a|bbU(\theta)=U(\theta)^Tba|bbU(\theta)=\tilde{\theta}^Ta|b\tilde{\theta}$ where $\tilde{\theta}=bU(\theta)$. Similarly, if $R_{n+1}=V$, then the new cuts come from $V(b)=a|b$. Since any word begins with $a$ after applying $U$ and since $V(b)a|b=aba|b$, we see that any indeterminate cut in the new sequence comes from $V(ab)=aa|b$. If the $ab$ is not in the last $\alpha_n\beta_n$, then by the induction hypothesis extends to a good cut $b\theta^Ta|b\theta a$ an since $\alpha_n\beta_n$ does not ends with $a$, it becomes $V(b\theta^Ta|b\theta a)a=abV(\theta^T)aa|bV(\theta)aa=abaV(\theta)^Ta|bV(\theta)aa$ which is good. 
		In the case that this $ab$ is located in the last $\alpha_n\beta_n$ and is indeterminate, it produces the only new indeterminate cut  $V(\theta^T)aa|bV(\theta)=aV(\theta)^Ta|bV(\theta)=a\tilde{\theta}^Ta|b\tilde{\theta}$ where $\tilde{\theta}=V(\theta)$.
	\end{proof}
	
	\begin{lemma}\label{lem:converge_alpha_beta}
		Let $\overline{R}_1,\overline{R}_2,\dots$ be a sequence of renormalization operators $\overline{U},\overline{V}$ such that both the operators $\overline{U},\overline{V}$ appears infinitely many times and let $(\alpha_0,\beta_0)=(a,b)$ and $(\alpha_{n+1},\beta_{n+1})=\overline{R}_{n+1}(\alpha_{n},\beta_{n})\in\overline{P}_{n+1}$ for $n\geq 0$.
		Notice that in this case $\lim_{n\to\infty}\alpha_n$ is a right infinite word and $\lim_{n\to\infty}\beta_n$ is a left infinite word.
		\begin{enumerate}
			\item 
			$\lim_{n\to\infty}\beta_n$ contains infinitely many indeterminate cuts in the suffix $\eta_n\theta_n^Ta|b\theta_n$, where $\theta_n$, $\eta_n$ are some finite words such that $\lim_{n\to\infty}\eta_n^T=\lim_{n\to\infty}\alpha_n$. \label{item:convergealpha}
			\item $\lim_{n\to\infty}\alpha_n$ contains infinitely many indeterminate cuts in the prefix $\theta_n^Ta|b\theta_n\eta_n$, where $\theta_n$, $\eta_n$ are some finite words such that $\lim_{n\to\infty}\eta_n^T=\lim_{n\to\infty}\beta_n$. \label{item:convergebeta}
		\end{enumerate}
	\end{lemma}
	
	\begin{proof}
		We only prove \Cref{item:convergealpha}, since the another case is analogous.
		
		Since both the operators $\overline{U},\overline{V}$ appear infinitely many times, we can pick a subsequence $\{n_k\}_{k\geq1}$ such that $\overline{U}$ appear in the first $n_1$ renormalization operators and that $\overline{R}_{n_k+1}=\overline{V}$ for any $k\geq 1$.
		Then $\beta_{n_k+1}=\alpha_{n_k}\beta_{n_k}$ a suffix of $\lim_{n\to\infty}\beta_n$ satisfying Case $1$ of \Cref{lem:extend_alphabeta}.
		So $\beta_{n_k+1}=\eta_{n_k}\theta_{n_k}^Ta|b\theta_{n_k}$ and the same cut can be extended to some subword $(\theta_{n_k}^\prime)^Ta|b\theta_{n_k}^\prime$ inside $\beta_{n_k+1}\alpha_{n_k+1}=\alpha_{n_k}\beta_{n_k}\alpha_{n_k}$.
		Thus $\eta_{n_k}^T$ and $\alpha_{n_k}$ shares a common prefix of length $|\theta_{n_k}^\prime|-|\theta_{n_k}|$, which can be arbitrary large by \Cref{lem:extend_alphabeta}.
		So we have that $\lim_{k\to\infty}\eta_{n_k}^T=\lim_{n\to\infty}\alpha_n$.
	\end{proof}

	\section{Proof of \texorpdfstring{\Cref{thm:full_characterization}}{Theorem 1.2} and \texorpdfstring{\Cref{thm:equal3_is_tran}}{Theorem 1.2}}\label{sec:characterization}

	
	\begin{proof}[Proof of \Cref{thm:equal3_is_tran} assuming \Cref{thm:full_characterization}]
		Without loss of generality we can assume $x\in[0,1)$. By \Cref{thm:full_characterization}, we know there is $\underline{y}=(y_n)_{n\in\Z}\in m^{-1}(3)\cap l^{-1}(3)$ such that $[0;y_1,y_2,\dots]=[0;x_{N+1},x_{N+2},\dots]$. Since $y_1y_2\dots$ is a (one-sided) balanced sequence (see \cite{reutenauer}), after applying the substitution $\chi:a\mapsto 22,b\mapsto 11$ it will have linear complexity growth (in fact, it is a quasi-sturmian sequence and will have a complexity function equal to $n+3$ for all $n\geq 4$ by \cite[Theorem 3.1]{Heinis}). Since adding the finite word $x_1\dots x_N$ to the right will only increase the complexity by a constant factor of at most $N$, we conclude that $x_1\dots x_Nx_{N+1}x_{N+2}\dots$ has linear complexity growth.

		Finally, since algebraic real numbers of degree at least $3$ do not have linear complexity (see \cite[Theorem 1.1]{transcendence}) and quadratic ones have bounded complexity, we conclude that $x=[0;x_1,x_2,\dots]$ is transcendental.
	\end{proof}
	
	\begin{remark}
		In general if $\underline{x}\in\{1,2\}^\mathbb{Z}$ is such that $m(\underline{x})=\sup_{n\in\Z}\lambda_n(\underline{x})\leq 3$, then $[0;x_1,x_2,\dots]$ is always a quadratic irrational (if $\underline{x}$ is eventually periodic) or transcendental.
	\end{remark}

	\subsection{Proving that the construction works}
	\begin{lemma}\label{lem:periodic_good}
		Let $N\in\N$, $(\alpha,\beta)\in\overline{P}$ and $x=[x_0;x_1,x_2,\dots]$ an irrational satisfying any one of the following
		\begin{equation}
			\begin{aligned}
				x_{N+1}x_{N+2}x_{N+3}\ldots&=(\beta^T)^\infty, \qquad &x_{N+1}x_{N+2}x_{N+3}\ldots&=\beta^T(\alpha^T)^\infty,  \\ x_{N+1}x_{N+2}x_{N+3}\ldots&=2\alpha^{+}\alpha^\infty, \qquad
				&x_{N+1}x_{N+2}x_{N+3}\ldots&=2\alpha^{+}\beta^\infty.
			\end{aligned}
		\end{equation}
		If $N\geq 1$ and $x_{N+1}x_{N+2}\dots\neq a^\infty, b^\infty$, then let $x_{N+1}x_{N+2}\dots=\theta^Tb|a\theta R$ or $2x_{N+1}x_{N+2}\dots=\theta^Ta|b\theta R$ be the unique such factorization determined by \Cref{lem:only_one_indeterminate_cut}, where $\theta$ has even length.
		Moreover, suppose that $[0;x_N,\dots,x_1]<[0;R]$or $[0;x_N,\dots,x_1]<[0;2,R]$ whenever $x_{N+1}=1$ or $x_{N+1}=2$, respectively. Then $\lambda_n(x)<3$ for all $n\geq N+1$ and so the inequality $\left| x-\frac{p}{q} \right| < \frac{1}{3q^2}$ has at most $N$ solutions.
	\end{lemma}
	
	\begin{proof}
		Recall by Legendre's theorem that all solutions of the inequality are of the form $p/q=p_n/q_n$, so it suffices to see the numbers of $n\geq 1$ with $\lambda_n(x)>3$.
		
		By \Cref{lem:only_one_indeterminate_cut}, all cuts inside $x_{N+1}x_{N+2}\dots$ are good except for the exceptional cut $x_{N+1}x_{N+2}\dots=\theta^Tb|a\theta R$ or $2(\theta^T)^+a|b\theta R$, with $\theta$ of even length, or $x_{N+1}x_{N+2}\dots=2|2a\dots$. 
		
		For the case $2|2a\dots$ we have
		\begin{align*}
			\lambda(x_1\dots x_N 2|2a\dots)=[a,x_N,\dots,x_1]+[0;a,\dots]<[a,x_N,\dots,x_1]+[0;b,x_N,\dots,x_1]=3,
		\end{align*}
		and if $x_{N+1}x_{N+2}\dots=a^\infty$, then this will be the only indeterminate cut.
		
		Now assume that $x_{N+1}x_{N+2}\dots\neq a^\infty,b^\infty$. 
		By \Cref{lem:only_one_indeterminate_cut} there exists a unique such factorization.
		The hypothesis on $x_1,\dots,x_N$ is precisely to guarantee that if $x_{N+1}=1$
		\begin{align*}
			\lambda(x_1\dots x_N\theta^T b|a\theta R) =[0;b,\theta,x_N,\dots,x_1]+[a,\theta,R] <[0;b,\theta,R]+[a,\theta,R]= 3,
		\end{align*}
		and if $x_{N+1}=2$
		\begin{align*}
			\lambda(x_1\dots x_N2(\theta^T)^+ a|b\theta R)=[a,\theta^-,2,x_N,\dots,x_1]+[0;b,\theta,R]<[a,\theta,R]+[0;b,\theta,R]=3.
		\end{align*}
		Therefore the only possible positions with $\lambda_n(x)>3$ are $1\leq n\leq N$.
	\end{proof}
	
	\begin{lemma}
		Suppose that the sequence of alphabets $(\alpha_{n+1},\beta_{n+1})\in \{\overline{U}(\alpha_{n},\beta_{n}),\overline{V}(\alpha_{n},\beta_{n})\}$ is such that both renormalization operators $\overline{U}$ and $\overline{V}$ appear infinitely many times. 
		Let $N\in\N$ and $x=[x_0;x_1,x_2,\dots]$ be an irrational such that
		\begin{enumerate}
			\item if $x_{N+1}=1$ then $x_{N+1}x_{N+2}x_{N+3}\ldots=\lim_{n\to\infty}\beta_n^T$and if $N\geq 1$ then $[0;x_N,\dots,x_1]<[0;\lim_{n\to\infty}\alpha_n]$; or
			
			\item if $x_{N+1}=2$ then $x_{N+1}x_{N+2}x_{N+3}\ldots=\lim_{n\to\infty}2\alpha_n^{+}$ and if $N\geq1$ then $[0;x_N,\dots,x_1]<[0;2,\lim_{n\to\infty}\beta_n^T]$. 
		\end{enumerate}	 
		Then $\lambda_n(x)<3$ for all $n\geq N+1$ and so the inequality $\left| x-\frac{p}{q} \right| < \frac{1}{3q^2}$ has at most $N$ solutions.
	\end{lemma}
	\begin{proof}
		Similar with the proof of \Cref{lem:periodic_good}, we focus on the numbers of $n\geq 1$ with $\lambda_n(x)>3$.
		Let $\underline{y}\in\{a,b\}^\Z$ be a bi-infinite sequence defined as follows: put $y_{N+1}y_{N+2}\dots=x_{N+1}x_{N+2}\dots$ and if $x_{N+1}=1$ then $\dots y_{N-1}y_N=\lim_{n\to\infty}\alpha_n^T$ and if $x_{N+1}=2$ then $\dots y_{N-1}y_N=\lim_{n\to\infty}\beta_n2$. In particular we have that the Markov value of $\underline{y}$ is exactly 3. 
		
		By \Cref{lem:safe_cuts} we have that all cuts inside the word $x_{N+1}x_{N+2}\dots$ are good, except for the indeterminate ones that are all suffixes of $x_{N+1}x_{N+2}\dots$ of the form $\theta^Tb|a\theta R$ with $|\theta|$ even or $\theta^Ta|b\theta R$ with $|\theta|$ odd, according to $x_{N+1}=1$ and $x_{N+1}=2$, respectively. 
		The hypothesis guarantees that in any case if $\{c,d\}=\{a,b\}$
		\begin{equation*}
			\lambda_n(x)=\lambda(x_1\dots x_N\theta^T c|d\theta R)<\lambda(\dots y_{N-1}y_{N}\theta^T c|d\theta R)=\lambda_n(\underline{y})\leq 3
		\end{equation*}
		Therefore the only possible positions with $\lambda_n(x)>3$ are $1\leq n\leq N$.
	\end{proof}

	\subsection{Characterization}
	
	We will assume through the rest of this section that $x=[x_0;x_1,x_2,\dots]$ is an irrational number such that the inequality $\left|x-\frac{p}{q}\right|<\frac{1}{3q^2}$ has only finitely many solutions $p/q\in\Q$. By \eqref{eq_n_th_convergent} this means there are only finitely many indices $n\geq 1$ such that 
	\begin{equation*}
		\lambda_n(x)=[x_n;x_{n+1},x_{n+2},\dots]+[0;x_{n-1},\dots,x_1]>3.
	\end{equation*}
	Take $N$ to be the maximum of those indices or let $N=0$ if there is none. Observe that $N=0$ precisely when $\widetilde{m}(x)\leq 3$. In particular, we have that there are no bad cuts in the suffix $x_{N+1}x_{N+2}\dots$ of the infinite word $x_1x_2\dots$. Observe that if $N\geq 1$, then by hypothesis $\lambda_N(x)>3$, so we must have $x_N\geq 2$ in that case.

	The next lemma states that all blocks of $x_{N+1}x_{N+2}\dots$ are even blocks of 1's and 2's except the first one, which, can only be an even block of 1's or an odd block of 2's. Its proof is similar to the descent argument of \cite[Lemma 9]{Bombieri}.

	\begin{lemma}\label{lem:evens_blocks}
		Let $x=[x_0;x_1,x_2,\dots]$ be an irrational and $N\in\N$ minimal such that $\lambda_n(x)\leq 3$ for all $n\geq N+1$.
		Then $x_n\in\{1,2\}$ for all $n\geq N+1$.
		Moreover either $x_{N+1}x_{N+2}\dots=1^{e_1}2^{f_1}1^{e_2}2^{f_2}\dots$ where all $e_i,f_i\geq 2$ are even or $x_{N+1}x_{N+2}\dots=2^{f_1}1^{e_1}2^{f_2}1^{f_2}\dots$ where $f_1\geq 1$ is odd and the rest $e_i,f_i\geq 2$ are even, where the number of blocks is finite if the last block is infinite.
	\end{lemma}
	\begin{proof}
		First, since $\lambda_n(x)\geq x_n$, it obviously holds that $x_n\in\{1,2\}$ for all $n\geq N+1$. Observe that $212$ is forbidden in $x_{N+1}x_{N+2}\dots$ since $[2;1,2]+[0;\overline{2,1}]>3$. Similarly, $x_{N+1}x_{N+2}\dots$ can not begin with $12\dots$ because if $N=0$ then $\lambda_2(x)\geq [2;1]+[0;x_3,\dots]>3$ and if $N\geq 1$, then $x_N\geq 2$ also gives that $\lambda_{N+2}(x)\geq [2;1,2]+[0;\overline{2,1}]>3$. Observe also that $121$ is forbidden in $x_{N+1}x_{N+2}\dots$ because $[2;1,1]+[0;1,1,\overline{2,1}]>3$.
		
		Now suppose that $x_{N+1}x_{N+2}\dots$ contains a subword of the form $1^\ell 2^m 1^r$ with $m\geq 1$ and $\ell$, $r$ maximal. By the previous observations $\ell\geq 2$ and $m\geq 2$. Note that if $w$ is a finite word (possibly empty), then we have
		\begin{equation*}
			\lambda(w 1^\ell |2^{m} 1^r \dots)=3+[0;1^\ell,w^T]-[0;1,1,2^{m-2},1^r,\dots].
		\end{equation*}
		Assuming $\ell$ odd, note that $\ell\geq 3$. Since the value of this cut must be less or equal to 3, then $m=2$ and $r-2\leq\ell$ is also odd.
		However this is not possible, because extending to the right it yields an infinite descent of positive odd numbers, so $\ell$ must be even. 
		Moreover, $\ell>2$ even implies that $m=2$. This argument also applies for $1^\ell 2^\infty$ and shows that $\ell=2$ in that case. 
		In conclusion all blocks of 1's appearing in $x_{N+1}x_{N+2}\dots$ are even. 
		
		Now assume that $12^\ell 1^m 2^r$ is a subword of $x_{N+1}x_{N+2}\dots$ with $m\geq1$ and $r$ maximal. Since $121$ is forbidden, one has that $\ell\geq 2$. If $\ell$ is odd, then for any finite word $w$ (possibly empty), the inequality
		\begin{equation*}
			\lambda(w 12^{\ell-2}2|21^m2^r\dots)=3+[0;1^m,2^r,1\dots]-[0;1,1,2^{\ell-2},1,w^T]<3,
		\end{equation*}
		shows that $m=2$ and that $r\leq\ell-2$ is also odd. The same descending argument as before gives a contradiction, so $\ell$ must be even. In the case of the subword $12^\ell1^\infty$, the same argument yields $\ell=2$. 
		
		Finally, let us show that $x_{N+1}x_{N+2}\dots$ can not begin with an even block of 2's. Indeed, suppose $x_{N+1}x_{N+2}\dots$ begins with $2^\ell1^m$ with $m$ maximal. If $N=0$, since we have already shown that $m$ must be even, if $\ell$ is even, we have
		\begin{equation*}
			\lambda(2^{\ell-1}|21^m\dots)=3+[0;1^m,\dots]-[0;1,1,2^{\ell-2}]>3,
		\end{equation*}
		because all blocks are even (or the last is infinite). In case that $N\geq 1$, then since $\lambda_N(x)>3$, one must have that either $x_N\geq 3$ or $x_N=2$ with $N\geq 2$ and $x_{N-1}=1$ (since otherwise we will have $\lambda_N(x)=\lambda(\dots2\vert22\dots)<3$). 
		In any case the same cut above is greater than 3, so we conclude that $\ell$ must be odd.
	\end{proof}
	
	Let $\omega:=x_{N+1}x_{N+2}\dots$ and $(\alpha_0,\beta_0)=(a,b)$. According to \Cref{lem:evens_blocks}, all blocks of 1's and 2's in $\omega$ are even, except the case when $\omega$ begins with an odd block of $2$.  
	\begin{itemize}
		\item In case that $x_{N+1}=1$, we can write $\omega$ as a word over the alphabet $(\alpha_0,\beta_0)=(a,b)$ and begins with $\beta_0$.
		\item In case that $x_{N+1}=2$, we can write the $2\omega$ as a word over the alphabet $(\alpha_0,\beta_0)=(a,b)$ and begins with $\alpha_0$. 
	\end{itemize}

	\subsubsection{Gurwood's proof of the case \texorpdfstring{$N=0$}{N=0}}\label{subsubsec:gurwood}
	
	Only in this subsection we will assume that $x=[x_0;x_1,x_2,\dots]$ is an irrational number such that $\left|x-\frac{p}{q}\right|\geq\frac{1}{3q^2}$ for all $p/q\in\Q$. In this subsection we will show how that Gurwood's characterization in lower and upper balanced sequences coincides with the case $N=0$ of \Cref{thm:full_characterization}. It is precisely this lemma that fails to generalize to the case of an arbitrary number of solutions, which is why we cannot extend Gurwood’s elegant proof to the general case.
	
	\begin{lemma}\label{lem:not_generalizable_lemma}
		Let $x=[x_0;x_1,x_2,\dots]$ be an irrational such that $\lambda_n(x)\leq 3$ for all $n\geq 1$ and let $\omega=x_{1}x_{2}\dots$. Then $x_{1}=1$ if and only if $\omega$ is an upper balanced sequence over $\{a,b\}$ and $x_{1}=2$ if and only if $2\omega$ is a lower balanced sequence over $\{a,b\}$.
	\end{lemma}
	
	\begin{proof}
		Recall that \Cref{lem:evens_blocks} allows us to regard $\omega$ or $2\omega$ as a word over $\{a,b\}$. We will prove the first case since the other is analogous. Since by hypothesis $\omega$ does not have bad cuts, any cut of $\omega$ extends to an either good or indeterminate cut. If $\omega$ contains a cut of the form $\omega=\theta^Ta|b\theta R$ with $\theta$ of even length, then this creates a bad cut because 
		\begin{equation*}
			\lambda(\theta^Ta|b\theta R)=[0;b,\theta,R]+[a;\theta]>[0;b,\theta,R]+[a;\theta,R]=3.
		\end{equation*}
		Thus $\omega$ must be upper balanced. On the other hand, if $\omega$ is already upper balanced, cuts of the form $\omega=\theta^Tb|a\theta R$ with $\theta$ of even length are good because
		\begin{equation*}
			\lambda(\theta^Tb|a\theta R)=[a;\theta,R]+[0;b,\theta]<[a;\theta,R]+[0;b,\theta,R]=3.
		\end{equation*}
	\end{proof}
	
	From the previous lemma and \Cref{lem:balanced_inequality} one gets a characterization of $x$. Let us write $\omega=w_1w_2\dots$ or $2\omega=w_1w_2\dots$ where $w_i\in\{a,b\}$ accordingly. From the previous lemma, we have our $w_1w_2\dots$ is upper balanced when $w_1=a$ and lower balanced when $w_1=b$. Recall that $w_1w_2\dots$ is a lower balanced sequence if and only if $bw_2\dots$ is an upper balanced sequence. 
	
	Now we show the equivalence between Gurwood's result with the case  $N=0$ of \Cref{thm:full_characterization}.
	First, let us suppose $w_1$ $w_1=b$ and that $\chi(w_n)=\lfloor n\xi \rfloor -\lfloor (n-1)\xi \rfloor$ for some $0\leq\xi<1$. If $\xi$ is rational. If $\xi=0$ then $w_1w_2\dots=a^\infty$ with corresponding upper balanced sequence $bw_2\dots=ba^\infty$.
	If $\xi>0$, then by \Cref{charcterization_xi} there is some $(u,v)\in\overline{P}$ such that $w_1w_2\dots=(uv)^\infty$, with the corresponding upper balanced sequence $bw_2\dots=u^bv(uv)^\infty=v^T((uv)^T)^\infty$ by \eqref{eq:periodic_identity2}.
	If $\xi$ is irrational, then by \Cref{charcterization_xi} there is a sequence $(\alpha_n,\beta_n)\in\overline{P}$ such that $w_1w_2\dots=\lim_{n\to\infty}\alpha_n$, with the corresponding upper balanced sequence $bw_2w_3\dots=\lim_{n\to\infty}(\beta_n)^T$.
	This is because $bw_2w_3\dots$ begins with $\alpha_{2n}^b=(\alpha_{2n-1}^{d_{2n-1}}\beta_{2n-1})^b(\alpha_{2n-1}^{d_{2n-1}-1}\beta_{2n-1})^{d_{2n}-1}=(\beta^T(\alpha_{2n-1}^T)^{d_{2n-1}-1}\alpha_{2n-1}^b)(\alpha_{2n-1}^{d_{2n-1}-1}\beta_{2n-1})^{d_{2n}-1}$, in particular, begins with $\beta_{2n-1}^T.$
	
	Now assume that  $w_1w_2\dots$ satisfies instead \eqref{eq:skew_xi}, that is, $w_1=a$ and $\chi(w_n)=\lfloor -(n-1)\xi \rfloor - \lfloor -n\xi \rfloor$ for $n\geq 2$. If $\xi$ is irrational this reduces to the previous case because $\lfloor n\xi\rfloor-\lfloor(n-1)\xi\rfloor=\lfloor -(n-1)\xi \rfloor - \lfloor -n\xi \rfloor$ for $n\geq 2$, since $\lfloor m\xi\rfloor+\lfloor -m\xi\rfloor=-1$ whenever $m\xi$ is not an integer. 
	If $\xi$ is rational, then since $\lfloor m\xi\rfloor+\lfloor -m\xi\rfloor=-1$ for $1\leq m\leq q-1$ and $\lfloor q\xi\rfloor+\lfloor -q\xi\rfloor=0$, by the previous case there exists some $(u,v)\in\overline{P}$ such that $bw_2\dots=((uv)_a^b)^\infty=((uv)^T)^\infty$, and by \eqref{eq:periodic_identity1} we have $w_1w_2\dots = u(uv)^\infty$.

	\subsubsection{Proof in the general case}

	Our proof of the general case uses nothing from balanced sequences. Instead, it is a direct proof based on renormalization and the properties of Christoffel words \Cref{subsec:properties}). We will need the following \cite[Lemma 3.17]{EGRS2024}.
	
	\begin{lemma}\label{lem:symmetry3}
		Let $(u, v) \in \overline{P}$ and let $e_1, \dotsc, e_k \geq 1$. If $(\alpha, \beta) = (u, uv)$, then
		\[
		u^b\beta\alpha^{e_1} \beta \alpha^{e_2} \beta \ldots \alpha^{e_k} v_a=(\alpha^{e_k}\beta\alpha^{e_{k-1}}\beta\ldots\beta\alpha^{e_1}\beta\beta)^T,
		\]
		while if $(\alpha, \beta) = (uv, v)$, then
		\[
		u^b\beta^{e_1}\alpha\beta^{e_2}\alpha \ldots \beta^{e_k}\alpha v_a = (\alpha\alpha \beta^{e_k} \alpha \beta^{e_{k-1}} \alpha \ldots \alpha \beta^{e_1})^T.
		\]
	\end{lemma}
	
	The proof of the following lemma is analogous to a part of the proof of \cite[Lemma 3.15]{EGRS2024}. In the mentioned lemma, it was proved that if a word can be written over the alphabet $(\alpha,\beta)\in\overline{P}$, begins with $\alpha\alpha$ and ends with $\beta\beta$, then a bad cut is forced to appear if we extend this word to the right over the alphabet $\{a,b\}$. The next statement will show that the same holds if we extend to the left over the alphabet $\{a,b\}$.
	
	\begin{lemma}\label{lem:leverage_cut}
		Let $(\alpha,\beta)\in\overline{P}$ and $w=\alpha\alpha(\beta\alpha)^k\beta\beta$ for some $k\geq 1$. Then 
		\begin{enumerate}
			\item It can be written in the form $w=|a\theta a|b\tau$ where
			\begin{enumerate}
				\item there exists a word $\gamma$ such that $\tau=\theta^Tab\gamma b$;
				\item $\theta ab\theta^Tab\gamma b$ begins with $\gamma a$.
			\end{enumerate}
			\item It can be written in the form $w=\tau a|b\theta b|$ where
			\begin{enumerate}
				\item there exists a word $\gamma$ such that $\tau=a\gamma^Tab\theta^T$;
				\item $\theta^Tba\theta ba\gamma a$ starts with $\gamma b$.
			\end{enumerate}
		\end{enumerate}
		Moreover, if $(\alpha,\beta)=(u,uv)$ then $\theta$ contains $v$ and if $(\alpha,\beta)=(uv,v)$ then $\theta$ contains $u$.  
		In particular in both cases $|\tau|\geq|\alpha\beta|$, and, if $w$ is a finite word over the alphabet $\{\alpha,\beta\}$ starting with $\alpha\alpha$ and ending with $\beta\beta$, then $w$ will contain such subwords.
	\end{lemma}
	
	\begin{proof}
		Since $(\alpha,\beta)\in\overline{P}$, we know that there is a $W=R_1\dots R_k$ where $R_i\in\{U,V\}$ such that $\alpha=W(a)$ and $\beta=W(b)$. Thus $w$ is the image of $aa(ba)^kbb$.
		
		The proof of the factorization $w=\tau a|b\theta b|$ is already contained in \cite[Lemma 3.15]{EGRS2024}, so we will only do the proof of the case $w=|a\theta a|b\tau $ (which is analogous). Write $w=\alpha\alpha(\beta\alpha)^k\beta\beta$ with $k\geq 0$ and $(\alpha,\beta)\in\overline{P}_n$.
		We will prove the lemma by induction on $n$. For $n=0$, by choosing $\gamma=(ab)^{k-1}$, $\theta=\varnothing$ and $\tau=(ab)^{k}b$, it can be easily verified that it satisfies the conclusion on the lemma. 
		
		We now show that the previous structure remains when we apply $U$ or $V$ to $w=|a\theta a|b\tau$.
		In the first case we have $U(w)=|abU(\theta)a|bbU(\tau)=|a\tilde{\theta}a|b\tilde{\tau}$ where $\tilde{\theta}=bU(\theta)$, $\tilde{\tau}=bU(\tau)$. We chose $\tilde{\gamma}=bU(\gamma)$. Therefore using the identities $bU(\theta^T)=U(\theta)^Tb, V(\theta^T)a=aV(\theta)^T$ from \Cref{lem:commutative_identities} we have
		\begin{equation*}
			\tilde{\tau}=bU(\theta^Tab\gamma b)=bU(\theta^T)abbU(\gamma)b=U)=U(\theta)^TbabbU(\gamma)b=\tilde{\theta}^Tab\tilde{\gamma}b.
		\end{equation*}
		Since $\theta ab\theta^Tab\gamma b=\gamma a\eta$ for some $\eta$, we have 
		\begin{equation*}
			\tilde{\theta}ab\tilde{\theta}^Tab\tilde{\gamma}b=bU(\theta)abU(\theta)^TbabbU(\gamma)b=bU(\theta ab\theta^Tab\gamma b)=bU(\gamma)abU(\eta)=\tilde{\gamma}ab\eta.
		\end{equation*}
		
		In the second case we have $V(w)=|aV(\theta)aa|bV(\tau)=|a\tilde{\theta}a|b\tilde{\tau}$ where $\tilde{\theta}=V(\theta)a$ and $\tilde{\tau}=V(\tau)$. We chose $\tilde{\gamma}=V(\gamma)a$. 
		\begin{equation*}
			\tilde{\tau}=V(\theta^Tab\gamma b)=V(\theta^T)aabV(\gamma)ab=aV(\theta)^TabV(\gamma)ab=\tilde{\theta}^Tab\tilde{\gamma}b.
		\end{equation*}
		Since $\theta ab\theta^Tab\gamma b=\gamma a\eta$ for some non empty $\eta$, we have
		\begin{equation*}
			\tilde{\theta}ab\tilde{\theta}^Tab\tilde{\gamma}b=V(\theta)aabaV(\theta)^TabV(\gamma)ab=V(\theta ab\theta^Tab\gamma b)=V(\gamma)aV(\eta)=\tilde{\gamma}V(\eta).
		\end{equation*}
		Since $\eta$ is non empty, $V(\eta)$ begins with $a$, so we are done.
		
		To prove the claim over $\theta$, notice that if $(\alpha,\beta)=(ab,b)$ then $\theta=b$ and if $(\alpha,\beta)=(a,ab)$ then $\theta=a$. The claim follows by induction since $\tilde{\theta}=bU(\theta)$ or $\tilde{\theta}=V(\theta)a$ according to whether we need to apply $U$ or $V$ to the pair $(\alpha,\beta)$.
		
		Finally, note that in general $|\tau|\geq|aab\theta|=|abb\theta|$, so we also have $|\tau|\geq|\alpha\alpha(\beta\alpha)^k\beta\beta|/2\geq|\alpha\beta|$
	\end{proof}
	
	Before proceeding, we introduce a technical lemma, which is very useful in the proof of the main result.
	
	\begin{lemma}\label{lem:force_continuation}
		Let $(\alpha,\beta)\in\overline{P}$ and $w=\alpha\alpha(\beta\alpha)^k\beta\beta$ for some $k\geq 1$.
		\begin{enumerate}
			\item If $w=|a\theta a|b\tau$ with $\tau=\theta^Tab\gamma b$, where $\theta,\tau,\gamma$ satisfy the Case 1 of \Cref{lem:leverage_cut}.
			Let $\gamma^\prime$ be a prefix of $\gamma$ over the alphabet $\{a,b\}$ and let $w^\prime=|a\theta a|b\tau^\prime$ with $\tau^\prime=\theta^Tab\gamma^\prime$. 
			Then any extension of $w^\prime$ to the left of length $|\gamma^\prime|+2$, is forced to be $(\gamma^\prime)^T bw^\prime$, otherwise the cuts marked in $w^\prime$ became bad for in the new extension. \label{item_force1}
			\item If $w=\tau a|b\theta b|$ with $\tau=a\gamma^Tab\theta^T$, where $\theta,\tau,\gamma$ satisfy the Case 2 of \Cref{lem:leverage_cut}.
			Let $\gamma^\prime$ be a prefix of $\gamma$ over the alphabet $\{a,b\}$ and let $w^\prime=\tau^\prime a|b\theta b|$ with $\tau^\prime=(\gamma^\prime)^Tab\theta^T$. Then any extension of $w^\prime$ to the right of length $|\gamma^\prime|+2$, is forced to be $w^\prime a\gamma^\prime$, otherwise the cuts marked in $w^\prime$ became bad for in the new extension. \label{item_force2}
		\end{enumerate}
	\end{lemma}
	
	\begin{proof}
		We give a proof for \Cref{item_force1}, the proof for \Cref{item_force2} is completely analogous.
		
		Assume there is an extension $w_{-(|\gamma^\prime|+2)}\dots w_{-1}$ such that $w_{-(|\gamma^\prime|+2)}\dots w_{-1}w$ has no bad cuts. 
		Since $\tau$ begins with $\theta^Tab$, the second cut forces a $b$ to the left of $w$, that is $w_{-1}=w_{-2}=1$. 
		In particular the configuration $b|a\theta a|b\tau^\prime$ is preceded by a $(\gamma^\prime)^T$: each $b$ of $(\gamma^\prime)^T$ is forced by the second cut (since $\tau^\prime$ begins with $\theta^Tab\gamma^\prime$) and each $a$ of $(\gamma^\prime)^T$ is forced by the first cut (since $\theta ab\theta^Tab\gamma^\prime$ begins with $\gamma^\prime$). 
		Finally we arrive to $(\gamma^\prime)^Tb|a\theta a|b\tau$. 
	\end{proof}
	
	The following lemma is a corollary of \Cref{lem:force_continuation}, which provides a criterion to determine when a bad cut exists.
	
	\begin{lemma}\label{lem:bad_extensions}
		Let $(\alpha,\beta)\in\overline{P}$. If $w$ is a finite word over the alphabet $\{\alpha,\beta\}$ starting with $\alpha\alpha$ and ending with $\beta\beta$. 
		Then any extension to the left or right of $w$ over the alphabet $\{a,b\}$ of the same length as $w$, will force a bad cut inside $w$.
	\end{lemma}
	\begin{proof}
		Without loss of generality we can assume that the finite word $w$ is of form $\alpha\alpha(\beta\alpha)^k\beta\beta$ with $k\geq 0$.
		For $k=0$, the lemma holds since we can always write $\alpha\alpha\beta\beta$ as $a\theta ab\theta b$ with $\theta$ palindrome, and will produce bad cut precisely at $w=a\theta a\vert b\theta b$.
		
		Now we consider the case $k\geq 1$.
		By \Cref{lem:force_continuation} the word $\omega=|a\theta a|b\tau$ is preceded by $\gamma^T$.
		In particular we arrive to $\gamma^Tb|a\theta a|b\tau$. 
		Finally, observe that the word $a\gamma^Tb|a\theta a|b\tau$ has a bad cut in the second cut (since $\tau$ begins with $\theta^Tab\gamma b$) and the word $b\gamma^Tb|a\theta a|b\tau$ has a bad cut in the first cut (because $\theta ab\theta^Tab\gamma b$ begins with $\gamma a$).
	\end{proof}

	Now we are ready to prove the main steps in the characterization. First we introduce several technical lemmas, which will include two parallel statements, corresponding to the cases starting with letter $1$ and $2$ respectively. We will only prove for one case, since the another is analogous.
	
	\begin{lemma}\label{lem:good_or_indeterminate}
		\begin{enumerate}
			\item Suppose that $\omega^T=\dots\omega_3\omega_2\omega_1$ $(\omega_i\in\{1,2\})$ can be written over the alphabet $(u,v)\in\overline{P}$ and  with $\omega_1=1$. Given $\eta a$ a suffix of $v_a$ or $u_a$, the cut $\eta a\vert\omega$ is good or indeterminate, moreover if the cut $\eta a\vert\omega$ is indeterminate, then $\omega$ will begin with $b\eta^T$. \label{item_a}
			\item Suppose that $\omega=\omega_1\omega_2\dots$ $(\omega_i\in\{1,2\})$ can be written over the alphabet $(u,v)\in\overline{P}$ and with $\omega_1=2$. Given $\eta b$ a suffix of $u^{+}$ or $v^{+}$, the cut $\eta b\vert\omega$ is good or indeterminate, moreover if the cut $\eta b\vert\omega$ is indeterminate, then $\omega$ will begin with $a\eta^T$. \label{item_b}
		\end{enumerate}
	\end{lemma}	
	\begin{proof}
		We give a proof of \Cref{item_a}, the proof for \Cref{item_b} is completely analogous. Moreover, we will only do the proof when $a\eta$ is a suffix of $v_a$ since the situation for $u_a$ is very similar. 
		
		First notice that $\omega$ will begin with $u^T$ or $v^T$.
		If $\omega$ begins with $v^T$, then we have $\eta a\vert \omega=\eta a\vert v^T\dots$ which is indeterminate because by \Cref{lem:symmetry1} we have $v=a\theta b$ for some palindrome $\theta$ (if $v=b$ then $\eta$ must be empty), and $\eta a$ is a suffix of $v_a=a\theta a$.
		If $\omega$ begins with $u^T$, that is, $\eta a\vert \omega=\eta a\vert u^T\dots$.
		We have two subcases, if $(u,v)=(\tilde{u}\tilde{v},\tilde{v})$ then we have $\eta a\vert \omega=\eta a\vert\tilde{v}^T\tilde{u}^T\dots$ which is indeterminate because $\eta a$ is a suffix of $v_a=\tilde{v}_a$.
		If $(u,v)=(\tilde{u},\tilde{u}\tilde{v})$, notice that since $\omega$ begins with $b$, one must have $\tilde{u}\neq a$.
		By \Cref{lem:symmetry2} we have $\tilde{u}_a^b=\tilde{u}^T$ and $v_a=\tilde{v}_a\tilde{u}_a^b=\tilde{v}_a\tilde{u}^T$. 
		Moreover if $|\eta a|\geq |\tilde{u}|$, we have $\eta a\vert \omega=\dots \tilde{u}^T\vert \tilde{u}^T\dots$ which is a good cut since $\tilde{u}=a\tilde{\theta}b$ for some palindrome $\tilde{\theta}$ by \Cref{lem:symmetry1}.
		If $|\eta a|< |\tilde{u}|$, then $\eta a\vert \omega= \eta a\vert \tilde{u}^T\dots$ is indeterminate since $\eta a$ is a suffix of $\tilde{u}^T$.
		
		Finally, the cut $\eta a\vert \omega$ is of form $\eta a\vert b\eta^T\dots$ for indeterminate case follows directly from the above analysis.
	\end{proof}
	\begin{remark}
		In the second item we can not change $v^{+}$ by $v^b$ because $\omega$ could begin with $uv=v_au^b$ which would make the cut $v^b|\omega=v^b|v_a\dots$ bad.
	\end{remark}
	
	\begin{lemma}\label{lem:bad_cut_truncated1}
		Let $x=[x_0;x_1,x_2,\dots]$ be an irrational and $N\in\N$ minimal such that $\lambda_n(x)\leq 3$ for all $n\geq N+1$. 
		\begin{enumerate}
			\item Suppose $\omega^T=\dots x_{N+3}x_{N+2}x_{N+1}$ $(x_i\in\{1,2\})$ with $x_{N+1}=1$ is a left infinite word that can be written over some $(u,v)\in\overline{P}$. 
			Then $\omega$ can not begin with  $(\theta^\prime)^Tab\theta b$            
			where $\theta,\theta^\prime$ are finite words over the alphabet $\{a,b\}$ such that $|\theta^\prime|\leq|\theta a|$, and $\theta a$ a prefix of $\theta^\prime v_a$ or $\theta^\prime u_a$. \label{item:bad_cut_truncated1}
			\item Suppose $\omega=x_{N+1}x_{N+2}x_{N+3}\dots$ $(x_i\in\{1,2\})$ with $x_{N+1}=2$ is a right infinite word that can be written over some $(u,v)\in\overline{P}$.
			Then $2\omega$ can not begin with  $(\theta^\prime)^Tba\theta a$ 
			where $\theta,\theta^\prime$ are finite words over the alphabet $\{a,b\}$ such that $|\theta^\prime|\leq|\theta b|$, and $\theta b$ a prefix of  $\theta^\prime u^b$ or $\theta^\prime v^b$.\label{item:bad_cut_truncated2}
		\end{enumerate}
	\end{lemma}
	\begin{proof}
		We give a proof of \Cref{item:bad_cut_truncated1}, the proof for \Cref{item:bad_cut_truncated2} is analogous.
		
		Suppose by contradiction that $\omega$ begins with such a subword. Let $w\in\{u,v\}$ be such that $\theta a$ is a prefix of $\theta^\prime w_a$. Notice that $\theta^\prime$ ends with $b$, so $|\theta^\prime|<|\theta a|$ and $\theta^\prime a$ is a prefix of $\theta$ since $w_a$ begins with $a$. Since $|\theta^\prime 2|$ is odd, if $N=0$ since $\theta^\prime$ is a prefix of $\theta$ we have
		\begin{equation}
			\begin{aligned}
				\lambda(\omega^T=(\theta^\prime)^T a|b\theta b\dots)&=[0;b,\theta,b,\dots]+[2;2,\theta^\prime] \\
				&>[0;b,\theta,b,\dots]+[2;2,\theta,a]=\lambda(a\theta^T a|b\theta b\dots)>3,
			\end{aligned}
			\label{eq_argument_bad_cut_1}
		\end{equation}
		which is a bad cut. Now let us consider the case $N\geq 1$. We have $x_N\leq 2$ since otherwise if $x_N\geq 3$, using the fact that $\theta^\prime$ is a prefix of $\theta$ of even length, one has
		\begin{equation}\label{eq:x_N<3}
			\lambda(x_1\dots x_N(\theta^\prime)^T a|b\theta b\dots)=3+[0;b,\theta,b,\dots]-[0;b,\theta^\prime,x_N,\dots,x_1]>3,
		\end{equation}
		which is a bad cut.
		On the other hand, we have $x_N\geq 2$ since $x_1\dots\vert x_N\omega$ is bad.
		So we have $x_N=2$, whence $N\geq 2$ and $x_{N-1}\leq 2$, because otherwise if $x_{N-1}\geq 3$ we have
		\begin{equation}\label{eq:x_N-1=2}
			\lambda_N(x)=\lambda(x_1\dots x_{N-1}|2bR)=[2;x_{N-1},\dots,x_1]+[0;b,R]<[2;2,R^T]+[0;b,R]=3.
		\end{equation}
		By the same trick in \eqref{eq:x_N<3} for the cut $x_1\dots x_N(\theta^\prime)^T a\vert b\theta b\dots$ and the fact that $\theta^\prime a$ is a prefix of $\theta$,  we have $x_{N-1}=2$.
		
		Now let $\eta=x_ix_{i+1}\dots x_{N-2}$ be the largest (possibly empty) common suffix between $(w_a)^{-}$ and $x_1\dots x_{N-2}$. 
		Notice that $x_1\dots x_N$ does not ends with $2w_a^{+}$, because otherwise, since $\theta a$ is a prefix of $\theta^\prime w_a$, one has that
		\begin{align*}
			\lambda(x_1\dots x_N(\theta^\prime)^Ta|b\theta b \dots)=\lambda(\dots 2w_a^{+}(\theta^\prime)^Ta|b\theta b \dots)=\lambda(\dots 2\theta^T a|b\theta b \dots) > 3,
		\end{align*}
		gives a bad cut inside $\omega^T$, contradicting the hypothesis. In particular $|\eta a|\leq|w_a|-2$. 
		
		If $[a,\eta^T,x_{i-1},...,x_1]\leq [w_a]$, since $|\eta a|\leq|w_a|-2$ then $[a,\eta^T,x_{i-1},...,x_1]<[a,\eta^T,\tilde{\eta}^T,R]$ for some suffix $\tilde{\eta}\eta a$ of $w_a$ and any infinite word $R$. 
		By hypothesis we have $\lambda(\tilde{\eta}\eta a|\omega)\geq \lambda(x_1\dots x_{i-1}\eta a|\omega)=\lambda_N(x)>3$, and, by the first item of \Cref{lem:good_or_indeterminate} we must have that $\omega$ begins with $b\eta^T\tilde{\eta}^T$.
		Hence
		\begin{align*}
			\lambda_N(x)&=\lambda(x_1\dots x_{i-1}\eta a|b\eta^T\tilde{\eta}^TR)\\
			&=[a,\eta^T,x_{i-1},\dots,x_1]+[0;b,\eta^T,R]<[a,\eta^T,\tilde{\eta}^T,R]+[0;b,\eta^T,\tilde{\eta}^T,R]=3,
		\end{align*}
		a contradiction. Therefore $[a,\eta^T,x_{i-1},\dots,x_1]>[w_a]$. However, this implies that we have a bad cut inside $\omega$ because
		\begin{align*}
			\lambda(x_1\dots x_{i-1}\eta a(\theta^\prime)^T a|b\theta b\dots)&=[a,\theta^\prime,a,\eta^T,x_{i-1},\dots,x_1]+[0;b,\theta,b,\dots] \\
			&>[a,\theta^\prime,w_a]+[0;b,\theta,b,\dots] \\
			&=\lambda(\dots a\theta^T a|b\theta b\dots)>3,
		\end{align*}
		where we used that $a\theta^\prime$ has even length and $\theta a$ is a prefix of $\theta^\prime w_a$.
	\end{proof}

	Before proceeding, we state two identities that are very useful which follow from \Cref{lem:symmetry3}. Let $k\geq 0$ and $\ell\geq 2$. If $(\alpha,\beta)=(uv,v)$ for some $(u,v)\in\overline{P}$, then  
	\begin{equation}\label{eq:identity1}
		u^b\beta^\ell(\alpha\beta)^k\alpha\alpha=u^b\beta^\ell(\alpha\beta)^k\alpha v_au^b=(\alpha\alpha(\beta\alpha)^k\beta^\ell)^Tu^b.
	\end{equation}
	If $(\alpha,\beta)=(u,uv)$, then 
	\begin{equation}\label{eq:identity2}
		\beta\beta(\alpha\beta)^k\alpha^\ell v_a=v_au^b\beta(\alpha\beta)^k\alpha^\ell v_a = v_a(\alpha^\ell(\beta\alpha)^k\beta\beta)^T.
	\end{equation}
	These identities will be used to either produce bad cuts or indeterminate cuts for subwords beginning with $\beta\beta$ and ending with $\alpha\alpha$.
	
	\begin{lemma}\label{lem:good_or_indeterminate2}
		Let $x=[x_0;x_1,x_2,\dots]$ be an irrational and $N\in\N$ minimal such that $\lambda_n(x)\leq 3$ for all $n\geq N+1$. 
		\begin{enumerate}
			\item Suppose $\omega^T=\dots x_{N+3}x_{N+2}x_{N+1}$ $(x_i\in\{1,2\})$ with $x_{N+1}=1$ is a left infinite word. If $\omega^T$ can be written over some alphabet $(\alpha,\beta)\in\overline{P}$, then $\omega$ can not contain $\beta\beta(\alpha\beta)^k\alpha\alpha$ for some $k\geq 0$. \label{item:good_or_indeterminate1}
			\item Suppose $\omega=x_{N+1}x_{N+2}x_{N+3}\dots$ $(x_i\in\{1,2\})$ with $x_{N+1}=2$ is a right infinite word. If $2\omega$ can be written over some alphabet $(\alpha,\beta)\in\overline{P}$, then $2\omega$ can not contain $\beta\beta(\alpha\beta)^k\alpha\alpha$ for some $k\geq 0$. \label{item:good_or_indeterminate2}
		\end{enumerate}
	\end{lemma}          
	
	\begin{proof}
		We give a proof of \Cref{item:good_or_indeterminate1}, the proof for \Cref{item:good_or_indeterminate2} is analogous.
		
		First assume that $\omega^T$ contains $\beta\beta(\alpha\beta)^k\alpha^{\ell}$ for some $k\geq 0$ and $\ell\geq 2$, where $(\alpha,\beta)=(u,uv)$. Without loss of generality we can further assume $\ell$ to be maximal. 
		In general we could only have two possible cases: $\omega^T$ contains $\beta\beta(\alpha\beta)^k\alpha^{\ell}\beta$, or $\omega^T$ ends with $\beta\beta(\alpha\beta)^k\alpha^{\ell}$.
		
		If $\omega^T$ contains $\beta\beta(\alpha\beta)^k\alpha^{\ell}\beta$, since $\beta=uv=v_au^b$, by \eqref{eq:identity2} one has
		\begin{equation*}
			\beta\beta(\alpha\beta)^k\alpha^\ell v_a=v_au^b\beta(\alpha\beta)^k\alpha^\ell v_a = v_a(\alpha^\ell(\beta\alpha)^k\beta\beta)^T.
		\end{equation*}
		This shows that $\alpha\alpha(\beta\alpha)^k\beta\beta$ is a subword of $\omega^T$ which contradicts \Cref{lem:bad_extensions} by extending to the right.
		
		Now assume that $\omega^T$ ends with $\beta\beta(\alpha\beta)^k\alpha^{\ell}$.
		The proof is by contradiction. First we consider the case $k=0$, that is, suppose $\omega^T$ ends with $\beta\beta\alpha^\ell$.
		Since we can always write $\alpha\alpha\beta\beta=a\theta a|b\theta b$ with $\theta$ palindrome, together with \eqref{eq:identity2} we have $\beta\beta\alpha^\ell v_a=v_a(\alpha^\ell\beta\beta)^T=v_a(\alpha^{\ell-2}a\theta a\vert b\theta b)^T$.
		Letting $v_a\tau^\prime=\alpha^{\ell-2}a\theta^T$ we have that $\omega$ will begin with $\tau^\prime a|b\theta b$.
		If $|\tau^\prime|> |a \theta|$, then $\omega^T$ will contain the subword $a\theta a|b\theta b$ which gives a bad cut. However, if $|\tau^\prime|\leq|a\theta|$, then $\theta a$ is a prefix of $(\tau^\prime)^T v_a$ and $(\tau^\prime)^T$ has even length, which contradicts \Cref{lem:bad_cut_truncated1}.

		For the case $k\geq 1$, the proof is analogous to the case $k=0$. 
		We can write $\alpha\alpha(\beta\alpha)^k\beta\beta=\tau a|b\theta b|$ where $\theta$ and $\tau$ satisfy Case 2 of \Cref{lem:leverage_cut}. 
		In particular, since $|\tau|\geq|\alpha\beta|$, using \eqref{eq:identity2} we can write $\beta\beta(\alpha\beta)^k\alpha^\ell v_a=v_a|b\theta^T b|a\tau^T(\alpha^T)^{\ell-2}=v_a|b\theta^T b|a(\tau^\prime)^T v_a$ and therefore $\omega$ begins with $\tau^\prime a|b\theta b|$ where $v_a\tau^\prime=\alpha^{\ell-2}\tau$.
		If $|\tau^\prime|\geq|\tau|$, then $\omega$ will have a subword of form $\tau a|b\theta b|$, which is precisely the word $\alpha\alpha(\beta\alpha)^k\beta\beta$,
		so by \Cref{lem:bad_extensions} there would exists a bad cut inside $\omega$ when extending to the right.
		Suppose $|\tau^\prime|<|\tau|$, that is, $\tau^\prime$ is a suffix of $\tau=a\gamma^Tab\theta^T$. We claim that $\omega$ begins with $\tau^\prime a|b (\tau^\prime)^T$.
		This is because if $|\tau^\prime|\leq|\theta ba|$, then the claim follows from the fact that $ab\theta^T$ is a suffix of $\tau$.
		If $|\tau^\prime|>|\theta ba|$, we can write $\tau^\prime=(\gamma^\prime)^Tab\theta^T$ with $\gamma^\prime$ a prefix of $\gamma$.
		By \Cref{lem:force_continuation} we know that $\omega$ will begin with $\tau^\prime a|b (\tau^\prime)^T =\tau^\prime a|b\theta b|a\gamma^\prime $. 
		Using the fact that $|\tau^\prime 2|$ is odd, if $N=0$ we conclude				
		\begin{equation}
			\begin{aligned}
				\lambda(\omega=\tau^\prime a|b (\tau^\prime)^T\dots)&= 		[0;b,(\tau^\prime)^T,R]+[2;2,(\tau^\prime)^T]  \\
				&>[0;b,(\tau^\prime)^T,R]+[2;2,(\tau^\prime)^T,R]=3,
			\end{aligned}
			\label{eq_argument_bad_cut_3}
		\end{equation}
		which is a contradiction.
		
		Now consider the case for $N\geq 1$. As before, let $\eta=x_ix_{i+1}\dots x_{N-2}$ be the largest (possibly empty) common suffix between $(v_a)^{-}$ and $x_1\dots x_{N-2}$. Since $x_1\dots x_N$ does not ends with $v_a$, we have that $|\eta a|<|v_a|$. By hypothesis the cut $\eta a|\omega$ can not be good, so by the first item of \Cref{lem:good_or_indeterminate} we must have that $\omega$ begins with $b\eta^T$. Exactly as in the proof of \Cref{lem:bad_cut_truncated1}, we can assume that $[a,\eta^T,x_{i-1},\dots,x_1]>[v_a]$. 
		We claim that this implies that $\omega$ begins with $\tau^\prime a|b\theta b|a\gamma$. 
		Indeed, we already know that it begins with $\tau^\prime a|b\theta b|a$ and now we use the same proof of \Cref{lem:force_continuation}: first, notice that by \Cref{lem:leverage_cut} we have $|\theta|\geq|v|$, so $|\tau^\prime ab\theta b|\geq |\tau ab\theta b|-|v_a|> |\tau|+|\theta|-|v|>|\gamma b|$, 
		in particular $\tau^\prime ab\theta b$ ends with $b\gamma^T$. 
		Now, if we extend $\tau^\prime a|b\theta b|a$ to the right, we will reach $\tau^\prime a|b\theta b|a\gamma$ since each $b$ of $\gamma$ is forced by the second cut and each $a$ of $\gamma$ is forced by the first cut and by the inequality
		\begin{align*}
			\lambda(x_1\dots x_{i-1}\eta a\tau^\prime a|b\theta ba R)&=[a,(\tau^\prime)^T,a,\eta^T,x_{i-1},\dots,x_1]+[0;b,\theta,b,a,R] \\
			&>[a,(\tau^\prime)^T,v_a]+[0;b,\theta,b,R] = \lambda(\alpha^{\ell-2}\tau a|b\theta b aR),
		\end{align*}
		where we used that $a\tau^\prime$ has even length. Finally, if we continue with either $a$ or $b$ to the right will give a bad cut inside $\omega$, because $\tau a|b\theta ba\gamma b=a\gamma^Tab\theta^Ta|b\theta ba\gamma b$ is a bad cut and $\tau^\prime ab\theta b|a\gamma a$ is also a bad cut because  $\theta^Tba(\tau^\prime)^T$ begins with $\gamma b$.
		
		Now assume that $\omega^T$ contains $\beta^\ell(\alpha\beta)^k\alpha\alpha$ for some $k\geq0, \ell\geq 2$ and $(\alpha,\beta)=(uv,v)$ for some $(u,v)\in\overline{P}$.
		Without loss of generality we can further assume $\ell$ to be maximal (possibly infinite). 
		In general we can only have two possible cases: $\omega^T$ contains $\alpha\beta^\ell(\alpha\beta)^k\alpha\alpha$, or $\omega^T$ contains $\beta^\infty(\alpha\beta)^k\alpha\alpha$. In any case, since $\alpha=uv=v_au^b$ and $\beta^\infty=(\beta^T)^\infty\alpha^b\beta^{\ell-1}$ by \eqref{eq:infinite_identities}, then $\omega^T$ contains $u^b\beta^\ell(\alpha\beta)^k\alpha v_au^b$. Hence by \eqref{eq:identity1}
		\begin{equation*}
			u^b\beta^\ell(\alpha\beta)^k\alpha\alpha=u^b\beta^\ell(\alpha\beta)^k\alpha v_au^b=(\alpha\alpha(\beta\alpha)^k\beta^\ell)^Tu^b,
		\end{equation*}
		so $\omega$ contains $\alpha\alpha(\beta\alpha)^k\beta\beta$, which again contradicts \Cref{lem:bad_extensions} by extending to the right.
	\end{proof}

	\begin{lemma}\label{lem:noalphaalphabetabeta}
		Let $x=[x_0;x_1,x_2,\dots]$ be an irrational and $N\in\N$ minimal such that $\lambda_n(x)\leq 3$ for all $n\geq N+1$. 
		\begin{enumerate}
			\item Suppose $\omega^T=\dots x_{N+3}x_{N+2}x_{N+1}$ $(x_i\in\{1,2\})$ with $x_{N+1}=1$ is a left infinite word. If $\omega^T$ can be written over some alphabet $(\alpha,\beta)\in\overline{P}$, then $\omega^T$ can not contain $\alpha\alpha$ and $\beta\beta$ simultaneously. \label{item:noalphaalphabetabeta1}
			\item Suppose $\omega=x_{N+1}x_{N+2}x_{N+3}\dots$ $(x_i\in\{1,2\})$ with $x_{N+1}=2$ is a right infinite word. If $2\omega$ can be written over some alphabet $(\alpha,\beta)\in\overline{P}$, then $2\omega$ can not contain $\alpha\alpha$ and $\beta\beta$ simultaneously.\label{item:noalphaalphabetabeta2}
		\end{enumerate}
	\end{lemma}
	\begin{proof}
		We give a proof of \Cref{item:noalphaalphabetabeta1}, the proof for \Cref{item:noalphaalphabetabeta2} is analogous.

		If $\omega^T$ contains a subword $w$ beginning in $\alpha\alpha$ and ending in $\beta\beta$, then \Cref{lem:bad_extensions} shows that an extension of $w$ to the left (that must be inside of $\omega^T$) contains bad cuts, a contradiction. 
		Now assume that it contains $w$ beginning with $\beta\beta$ and ending with $\alpha\alpha$. 
		If $(\alpha,\beta)=(a,b)$, then we can assume $\omega^T$ contains $bb(ab)^kaa$ for some $k\geq 0$, so the transpose $\omega=x_{N+1}x_{N+2}\dots$ contains $aa(ba)^kbb$ which contradicts \Cref{lem:bad_extensions} by extending to the right. The rest of the cases are ruled out by \Cref{lem:good_or_indeterminate2}.
		
	\end{proof}

	\begin{lemma}\label{lem:forbidden_pattern1}
		Let $x=[x_0;x_1,x_2,\dots]$ be an irrational and $N\in\N$ minimal such that $\lambda_n(x)\leq 3$ for all $n\geq N+1$. 
		\begin{enumerate}
			\item Suppose $\omega^T=\dots x_{N+3}x_{N+2}x_{N+1}$ $(x_i\in\{1,2\})$ with $x_{N+1}=1$ is a left infinite word. If $\omega^T$ can be written over some alphabet $(\alpha,\beta)=(uv,v)$ where $(u,v)\in\overline{P}$. 
			Then $\omega^T$ can not end with $\beta^\ell\alpha^k$ for some $k,\ell\geq 1$. \label{item:forbidden_pattern1}
			\item Suppose $\omega=x_{N+1}x_{N+2}x_{N+3}\dots$ $(x_i\in\{1,2\})$ with $x_{N+1}=2$ is a right infinite word. If $2\omega$ can be written over some alphabet $(\alpha,\beta)=(u,uv)$ where $(u,v)\in\overline{P}$.
			Then $2\omega$ can not begin with
			$\beta^k\alpha^\ell$ for some $k,\ell \geq 1$.\label{item:forbidden_pattern2}
		\end{enumerate}
	\end{lemma}
	
	\begin{proof}
		We give a proof of \Cref{item:forbidden_pattern1}, the proof for \Cref{item:forbidden_pattern2} is analogous.
		
		By \Cref{lem:noalphaalphabetabeta} we know that $\omega^T$ does not contain $\alpha\alpha$, $\beta\beta$ simultaneously.
		
		If $\omega^T$ contain $\beta\beta$, we can assume that $\omega^T$ end with $\beta^\ell(\alpha\beta)^k\alpha$ for some $k\geq 0,\ell\geq 2$, where $(\alpha,\beta)=(uv,v)$.
		Without loss of generality we can further assume $\ell$ to be maximal.
		In general we can only have two possible cases: $\omega^T$ ends with $\alpha\beta^{\ell}(\alpha\beta)^k\alpha$, or $\omega^T=\beta^{\infty}(\alpha\beta)^k\alpha$.
		
		First we consider the case $k=0$.
		We have $\omega$ beginning with $v^Tu^T|u^bvv$ since by \Cref{lem:symmetry2} we have $\omega=\alpha^T|(\beta^T)^\ell\alpha^T\dots=v^Tu^T|(v^T)^\ell u^bv_a\dots=v^Tu^T|u^bv^\ell v_a\dots$ or $\omega=\alpha^T|(\beta^T)^\infty=v^Tu^T|u^bv^\infty$. Since we can write $uv=a\hat{\theta} b$ for some palindrome $\hat{\theta}$, we see that $\omega$ begins with $b\hat{\theta} a|b\hat{\theta}bv^{-}b$ so $\theta a=\hat{\theta}bv^{-}a$ is equal to $\theta^\prime v_a=\hat{\theta}bv_a$. This contradicts \Cref{lem:bad_cut_truncated1}.

		Now we consider the case $k\geq 1$.
		We have $\omega$ beginning with $v^Tu^T|u^bvu$ since by \Cref{lem:symmetry2} if $k\geq 2$ we have $\omega=\alpha^T|(\beta^T\alpha^T)^k\dots=v^Tu^T|(v^Tv^Tu^T)^k\dots=v^Tu^T|u^bvv_au^b\dots=v^Tu^T|u^bvuv\dots$, and if $k=1$ we have $\omega=\alpha^T|\beta^T\alpha^T\beta^\ell\alpha\dots=v^Tu^T|u^bvv_au^bv^\ell\dots=v^Tu^T|u^bvuv^{\ell+1}\dots$ or $\omega^T=\alpha^T|\beta^T\alpha^T(\beta^T)^\infty=v^Tu^T|u^bvv_au^bv^\infty=v^Tu^T|u^bvuv^\infty$. In any case, we have that $\theta^\prime=u^{+}v$ is such that $\theta a=u^{+}vu_a$ is a prefix of $\theta^\prime u_a$ with $|\theta^\prime|\leq|\theta a|$, contradicting \Cref{lem:bad_cut_truncated1}.

		Now if $\omega^T$ does not contain $\beta\beta$, then we can assume $\omega^T$ end with $\alpha^t\beta\alpha^k$ with $t$ maximal (possibly infinite).
		In general we can only have two possible cases: $\omega^T=\alpha^\infty\beta\alpha^k$, or $\omega^T$ ends with $\alpha\beta\alpha^t\beta\alpha^k$.

		For $\omega^T=\alpha^\infty\beta\alpha^k$, we have
		\begin{equation*}
			\omega=(\alpha^T)^k|\beta^T(\alpha^T)^\infty=(v^Tu^T)^{k-1}v^Tu^T|u^bv(uv)^\infty=(v^Tu^T)^{k-1}v^Tu^T|u^bv(uv)^{k-1}u\dots,
		\end{equation*}
		so $\omega$ begins with $(\theta^\prime)^T a|b\theta b$ where $\theta^\prime=u^{+}v(uv)^{k-1}$ and $\theta b=u^{+}v(uv)^{k-1}u$, so $\theta a$ is equal to $\theta^\prime u_a$ and $|\theta^\prime|\leq|\theta a|$, a contradiction with \Cref{lem:bad_cut_truncated1}.

		If $\omega^T$ ends with $\alpha\beta\alpha^t\beta\alpha^k$, we have
		\begin{equation*}
			\omega=(\alpha^T)^k|\beta^T(\alpha^T)^t\beta^T\alpha^T\dots=(v^Tu^T)^{k-1}v^Tu^T|u^bv(uv)^tvv_a\dots.
		\end{equation*} 
		We must have $t\geq k-1$ because otherwise if $t\leq k-2$, we have
		\begin{equation*}
			\omega=(v^Tu^T)^{k-t-2}u^bv_a(v^Tu^T)^{t}v^Tu^T|u^bv(uv)^tvv_a\dots
		\end{equation*}
		containing $v_a(v^Tu^T)^{t}v^Tu^T|u^bv(uv)^tv$, which is a bad cut because the first letter of $v_a$ is $a$ while the last letter of $v$ is $b$.
		So we have $\omega$ beginning with $(v^Tu^T)^{k-1}v^Tu^T|u^bv(uv)^{k-1}v$ (when $t=k-1$) or $(v^Tu^T)^{k-1}v^Tu^T|u^bv(uv)^{k-1}u$ (when $t\geq k$). Hence $\omega$ begins with $(\theta^\prime)^T a|b\theta b$ where $\theta^\prime=u^{+}v(uv)^{k-1}$ and $\theta b=u^{+}v(uv)^{k-1}w$ with $w\in\{u,v\}$, so $\theta a$ is equal to $\theta^\prime w_a$ and $|\theta^\prime|\leq|\theta a|$, a contradiction with \Cref{lem:bad_cut_truncated1}. 
	\end{proof}

	\begin{lemma}\label{lem:renormalizable_limit1}
		Let $x=[x_0;x_1,x_2,\dots]$ be an irrational and $N\in\N$ minimal such that $x_{N+1}=1$, $\lambda_n(x)\leq 3$ for all $n\geq N+1$.
		If $x$ is ultimately periodic, then there is $(\alpha,\beta)\in\overline{P}$ such that $x_{N+1}x_{N+2}x_{N+3}\dots = \beta^T(\alpha^T)^\infty$ or $x_{N+1}x_{N+2}x_{N+3}\dots = (\beta^T)^\infty$. 
		Moreover, if if $N\geq 1$ and $x_{N+1}x_{N+2}\dots\neq b^\infty$, then we have a unique factorization $x_{N+1}x_{N+2}\dots=\theta^Tb|a\theta R$ such that $[0;x_N,\dots,x_1]<[0;R]$. 
		
		If $x$ is not ultimately periodic,  then there is a sequence of alphabets $(\alpha_{n+1},\beta_{n+1})\in \{\overline{U}(\alpha_{n},\beta_{n}),\overline{V}(\alpha_{n},\beta_{n})\}$ with both renormalization operators $\overline{U}$ and $\overline{V}$ appearing infinitely many times and such that
		\begin{equation}\label{eq:limit_beginning_1}
			x_{N+1}x_{N+2}x_{N+3}\dots = \lim_{n\to\infty}\beta_n^T.
		\end{equation}
		Moreover, if $N\geq 1$ then
		\begin{equation}\label{eq:sufficient_condition1}
			[0;x_N,\dots,x_1]<[0;\lim_{n\to\infty}\alpha_n].
		\end{equation}
	\end{lemma}
	\begin{proof}
		According to \Cref{lem:evens_blocks}, all blocks of 1's and 2's in $x_{N+1}x_{N+2}\dots$ are even. 
		So we can write $x_{N+1}x_{N+2}\dots$ as a word in $\{a,b\}$, say $bw_2w_3\dots$ with $w_i\in\{a,b\}$. 
		Write $\omega^T:=\dots w_3w_2b$, which is a word over the alphabet $(\alpha_0,\beta_0)=(a,b)$. We claim that for all $n\in\mathbb{N}$, the word $\omega^T$ can be written in some alphabet $(\alpha_n,\beta_n)\in\overline{P}_n$ and moreover it always ends with $\beta_n$. 
		
		We have already shown that this is true for the base case, so suppose $\omega^T$ can be written over the alphabet $(\alpha,\beta)=(\alpha_n,\beta_n)$ and ends with $\beta$. 
		The only obstructions to write $\omega^T$ in $\overline{U}(\alpha,\beta)$ or $\overline{V}(\alpha,\beta)$ and to end with $\beta_{n+1}$, is when $\omega^T$ contains both $\alpha\alpha$ and $\beta\beta$ or when it ends with $\alpha\beta$, contains no $\alpha\alpha$ and is different from $(\alpha\beta)^\infty$. 
		We will show that these configurations give bad cuts inside $\omega^T$ and thus are forbidden.

		By \Cref{lem:noalphaalphabetabeta}, we have that $\omega^T$ can not contain simultaneously $\alpha\alpha$ and $\beta\beta$.
		
		Now assume that $\omega^T$ ends with $\alpha\beta$, contains no $\alpha\alpha$ and is different from $(\alpha\beta)^\infty$.
		In particular, we can write $\omega^T$ over the alphabet $(\hat{\alpha},\hat{\beta})=(\alpha\beta,\beta)$ and will end with a word of from $\hat{\beta}^\ell\hat{\alpha}^k$ for some $k,\ell\geq 1$, while this is forbidden by \Cref{lem:forbidden_pattern1}.

		We have completed the proof of the claim.
		So if $\omega^T$ is not ultimately periodic, we will have both the renormalization operators $\overline{U},\overline{V}$ appearing infinitely many times and then $\omega=\lim_{n\to\infty}\beta_n^T$.
		
		If $\omega^T$ is ultimately periodic, then at some finite step of renormalization the period would be $\alpha$ or $\beta$. 
		So eventually $\omega^T$ will be of form $\alpha^\infty\beta\alpha^{k_m}\dots\alpha^{k_1}\beta$
		or $\beta^\infty\alpha\beta^{\ell_n}\dots\alpha\beta^{\ell_1}$ with $k_i,\ell_i\geq 1 $.
		For the first case we will have $m = 0$ because
		otherwise under some steps of renormalization $\omega^T$ will contain $\alpha\alpha$ and $\beta\beta$ simultaneously.
		This corresponds to the case that $\omega^T=\alpha^\infty\beta$.
		And for the second case by the same reason we will have $n=0$ or $1$, that is, $\omega^T=\beta^\infty$ or $\omega^T=\beta^\infty\alpha\beta^{\ell_1}$.
		However, by letting $(\tilde{\alpha},\tilde{\beta})=(\alpha\beta^{\ell_1},\beta)$ we have $\beta^\infty\alpha\beta^{\ell_1}=\tilde{\beta}^\infty\tilde{\alpha}$, which is again forbidden by \Cref{lem:forbidden_pattern1}.
		
		Finally, we will prove \eqref{eq:sufficient_condition1}. 
		Notice that by \Cref{lem:converge_alpha_beta} the word $x_{N+1}x_{N+2}\dots$ can be written as $\theta_m^Tb|a\theta_m\eta_m^TR$ for some finite word $\eta_m,\theta_m$ and infinite word $R$ such that $|\theta_m|$ even and $\lim_{m\to\infty}\eta_m^T=\lim_{n\to\infty}\alpha_n$.
		Therefore, since \[\lambda(x_1\dots x_N \theta_m^Tb|a\theta_m \eta_m^TR)<3=\lambda(R^T\eta_m\theta_m^T b|a\theta_m \eta_m^TR),\] we have that $[0;b,\theta_m,x_N,\dots,x_1]<[0;b,\theta_m,\eta_m^T,R]$ or equivalently $[0;x_N,\dots,x_1]<[0;\eta_m^T]$ and taking the limit proves \eqref{eq:sufficient_condition1}.
		
		In case that $x$ is eventually periodic, by  \Cref{lem:only_one_indeterminate_cut} we have a unique factorization $x_{N+1}x_{N+2}\dots=\theta^Tba\theta R$ for some word $\theta$ of even length. Therefore, since $\lambda(x_1\dots x_N \theta^Tb|a\theta R)<3=\lambda(R^T\theta^Tb|a\theta R)$, we have that $[0;b,\theta,x_N,\dots,x_1]<[0;b,\theta,R]$ or equivalently $[0;x_N,\dots,x_1]<[0;R]$.
		
	\end{proof}
	
	The proof of the following lemma is completely analogous to the proof of \Cref{lem:renormalizable_limit1}, so we will omit the proof.
	
	\begin{lemma}\label{lem:renormalizable_limit2}
		Let $x=[x_0;x_1,x_2,\dots]$ be an irrational and $N\in\N$ minimal such that $x_{N+1}=2$, $\lambda_n(x)\leq 3$ for all $n\geq N+1$.
		If $x$ is ultimately periodic, then there is $(\alpha,\beta)\in\overline{P}$ such that $x_{N+1}x_{N+2}\dots =2\alpha^{+}\alpha^\infty$ or $x_{N+1}x_{N+2}\dots =2\alpha^{+}\beta^\infty$. 
		Moreover, if if $N\geq 1$ and $x_{N+1}x_{N+2}\dots\neq a^\infty$, then we have a unique factorization $2x_{N+1}x_{N+2}\dots=\theta^Ta|b\theta R$ such that $[0;x_N,\dots,x_1]<[0;2,R]$. 
		
		If $x$ is not ultimately periodic,  then there is a sequence of alphabets $(\alpha_{n+1},\beta_{n+1})\in \{\overline{U}(\alpha_{n},\beta_{n}),\overline{V}(\alpha_{n},\beta_{n})\}$ with both renormalization operators $\overline{U}$ and $\overline{V}$ appearing infinitely many times and such that
		\begin{equation}\label{eq:limit_beginning_2}
			x_{N+1}x_{N+2}\dots=\lim_{n\to\infty}2\alpha_n^{+}.
		\end{equation}
		
		Moreover, if $N\geq 1$ then
		\begin{equation}\label{eq:sufficient_condition2}
			[0;x_N,\dots,x_1]<[0;2,\lim_{n\to\infty}\beta_n^T].
		\end{equation}

	\end{lemma}
	
	\section{Proof of \texorpdfstring{\Cref{thm:characterization_tildeM}}{Theorem 1.4}}\label{sec:tildeMbefore3}
	
	\subsection{Restatement of the result}
	
	As stated in \Cref{subsec:spectrum_widetilde_M_before_3}, we only need to prove the following.
	
	\begin{theorem}\label{thm:equal_value}
		Let $(u,v)\in\overline{P}$. Then
		\begin{equation*}
			\widetilde{m}\left(\left((uv)^T\right)^\infty\right)=\widetilde{m}\left(2u^+(uv)^\infty\right)=\widetilde{m}\left(v^T\left((uv)^T\right)^\infty\right)=\widetilde{m}\left(2u^+v(uv)^\infty\right)
		\end{equation*}
		and this common value is equal to
		\begin{equation*}
			\frac{3+m\left((uv)^\infty\right)}{2},
		\end{equation*}
		where $m\left((uv)^\infty\right)$ is the Markov value of the bi-infinite sequence $(uv)^\infty$.
	\end{theorem}

	\subsection{The Proof}
	
	In \cite[Theorem 27]{Bombieri}, it was proved that if $(u,v)\in\overline{P}$, then the Markov value of the bi-infinite periodic orbit $(uv)^\infty$ is attained precisely in two positions of the period $uv=a\theta b$ and it is equal to
	\begin{equation}\label{eq:markov_value_first_version}
		m((uv)^\infty)=\lambda((uv)^\infty uv|uv(uv)^\infty)=[\overline{a,\theta,b}]+[0;\overline{b,\theta,a}]=\frac{\sqrt{\Delta}}{q},
	\end{equation}
	where $\Delta=9q^2-4$ and $q=q(uv)$ is the lower--left corner of the matrix $M_{uv}$ defined on \eqref{eq:matrix_M_w}.
	
	Let us give an independent proof of \eqref{eq:markov_value_first_version} and of the fact that this Markov value is attained at two positions based purely on renormalization.
	
	\begin{lemma}\label{lem:cuts}
		Let $uv\in P$, $uv\neq ab$ and write $uv=v_au^b=a\theta b$ where $\theta$ is a palindromic word on $\{a,b\}$. Given any factorization of the bi-infinite word $(uv)^\infty=y^Tb|ax$ different from $\dots uv|uv\dots$ or $(uv)^\infty=x^Ta|by$ different from $\dots v_a|u^b\dots$, we have that
		\begin{enumerate}
			\item Let $\theta_1$ be the largest prefix common to $x$ and $\theta$. Then $\theta_1b$ is a prefix of $x$ and $\theta_1a$ is a prefix of $\theta$.
			\item $\theta_2$ be the largest prefix common to $y$ and $\theta$. Then $\theta_2a$ is a prefix of $y$ and $\theta_2b$ is a prefix of $\theta$.
		\end{enumerate}
		
	\end{lemma}
	
	\begin{proof}
		Since $uv=(R_1\dotsb R_n)(ab)$ for some renormalization operators $R_1,\dots,R_n\in\{U,V\}$, we can do the proof by induction. Since for $uv=aab$ or $uv=abb$ the bi-infinite word $(ab)^\infty$ only has two cuts, there is nothing to prove. Assume that the lemma is true and lets prove it for $\tilde{u}\tilde{v}=W(uv)$ where $W\in\{U,V\}$. All factorizations of $uv$ will be in correspondence to factorizations of $\tilde{u}\tilde{v}$, however there can be new factorizations that we will analyze separately. 
		
		We will use the following fact that is very easy to prove by induction: the largest block of $a$'s inside of $uv$ is the first block and the largest block of $b$'s inside of $uv$ is the last block. 
		
		If we apply $W=U$, then $\tilde{u}\tilde{v}=U(uv)=U(a\theta b)=abU(\theta)b=a\tilde{\theta}b$. Observe that the cut $(uv)^\infty=y^Tb|ax$ corresponds to the cut of $\tilde{u}\tilde{v}$ given by
		\begin{equation*}
			(\tilde{u}\tilde{v})^\infty=U((uv)^\infty)=U(y^Tb|ax)=U(y^T)b|abU(x)=\tilde{y}^Tb|a\tilde{x}.
		\end{equation*}
		Since $\theta$ and $x$ begin with $\theta_1a$ and $\theta_1b$ respectively, we see that $\tilde{\theta}=bU(\theta)$ and $\tilde{x}=bU(x)$ begin with $bU(\theta_1)ab$ and $bU(\theta_1)b$, respectively, so $\tilde{\theta}_1=bU(\theta_1)$ is the new common prefix. Similarly, since $y$ begins with $\theta_2a$, using \Cref{lem:commutative_identities} we have that $\tilde{y}=U(y^T)^T$ begins with
		\begin{equation*}
			\left(U\left(a\theta_2^T\right)\right)^T=\left(abU\left(\theta_2^T\right)\right)^T=\left(aU(\theta_2)^Tb\right)^T=bU(\theta_2)a,
		\end{equation*}
		and since $\tilde{\theta}=bU(\theta)$ begins with $bU(\theta_2)b$, we have that $\tilde{\theta}_2=bU(\theta_2)$ is the new common prefix. Analogously the cut $(uv)^\infty=x^Ta|by$ corresponds to the cut of $\tilde{u}\tilde{v}$ given by
		\begin{equation*}
			(\tilde{u}\tilde{v})^\infty=U((uv)^\infty)=U(x^Ta|by)=U(x^T)a|bbU(y)=\tilde{x}^Ta|b\tilde{y}.
		\end{equation*}
		Since $\theta_1b$ is a prefix of $x$, we have that $U(x^T)$ ends with $bbU(\theta_1^T)=bU(\theta_1)^Tb$ (there is a $b$ before $bU(\theta_1^T)$ because any nonempty word ends with $b$ after applying $U$), so $\tilde{x}=U(x^T)^T$ begins with $bU(\theta_1)b$. We have that $\tilde{y}=bU(y)$ begins with $bU(\theta_2)ab$ so $\tilde{\theta}_2=bU(\theta_2)$ is a again the common prefix between $\tilde{\theta}$ and $\tilde{y}$.
		
		This shows that the previous factorizations $\dots b|a\dots$ or $\dots a|b \dots$ of $(uv)^\infty$ correspond to factorizations of $(\tilde{u}\tilde{v})^\infty$ with the same properties when we apply $U$. However, there are new factorizations appearing at the cuts $\dots a|a\dots$ because $U(a|a)=ab|ab$. For this, let us write $uv=a^nbw$ with $w$ some word over the alphabet $\{a,b\}$, so in particular $U(\theta)$ begins with $(ab)^{n-1}b$. Given any subfactor $a^m$ of $uv$ with $m\geq 2$, we must have that $m\leq n$, so any new factorization that comes from $a^ka|aa^\ell b$ has the form
		\begin{equation*}
			\tilde{y}^Tb|a\tilde{x}=\dots a(ba)^kb|a(ba)^\ell bb \dots
		\end{equation*}
		with $k+\ell+2=m$. In particular for that cut it holds that $\tilde{y}$ begins with $a$ (so $\tilde{\theta_2}$ is the empty word) and $\tilde{x}$ begins with $\tilde{\theta}_1b=(ba)^\ell bb$. Since $\tilde{\theta}=bU(\theta)$ begins with $b(ab)^{n-1}b=(ba)^{n-1}bb$ and $\ell\leq m-2\leq n-2$, we are done. In conclusion all cuts of $(\tilde{u}\tilde{v})^\infty$ where $\tilde{uv}=U(uv)$ have the claimed properties.

		Now if we apply $W=V$ then one has that $\tilde{u}\tilde{v}=V(uv)=V(a\theta b)=aV(\theta)ab=a\Tilde{\theta}b$. Hence the factorization $(uv)^\infty=y^Tb|ax$ will change to 
		\begin{equation*}
			(\tilde{u}\tilde{v})^\infty=\left(V(uv)\right)^\infty=V(y^Tbax)=V(y^T)ab|aV(x)=\tilde{y}^Tb|a\tilde{x}.
		\end{equation*}
		Since $\theta$ begins with $\theta_1a$, we have that $\tilde{\theta}=V(\theta)a=V(\theta_1)aa\dots$ and since $x$ begins with $\theta_1 b$, one has $\tilde{x}=V(x)=V(\theta_1)ab\dots$ so $\tilde{\theta}_1=V(\theta_1)a$ is the new largest common prefix. Similarly, since 
		\begin{equation*}
			\tilde{y}=\left(V(y^T)a\right)^T=\left(aV(y)^T\right)^T=V(y)a
		\end{equation*}
		then $y$ begins with $\theta_2a$ and $\theta$ begins with $\theta_2 b$, so $\tilde{y}$ and $\tilde{\theta}$ begin with $V(\theta_2)aa$ and $V(\theta_2b)=V(\theta_2)ab$, respectively, so $\tilde{\theta}_2=V(\theta_2)a$ is the new largest common prefix. Analogously, the cut $(uv)^\infty = x^Ta|by$ becomes
		\begin{equation*}
			\tilde{u}\tilde{v}=\left(V(uv)\right)^\infty=V\left(x^Ta|by\right)=V(x^T)aa|bV(y)=\tilde{x}^Ta|b\tilde{y}.
		\end{equation*}
		We have that
		\begin{equation*}
			\tilde{x}=\left(V(x^T)a\right)^T=\left(aV(x)^T\right)^T=V(x)a
		\end{equation*}
		Since $\theta_1b$ is a prefix $x$ and $\theta_1a$ is a prefix of $\theta$, we have that $V(\theta_1)aba$ is a prefix of $\tilde{x}$  and $\tilde{\theta}=V(\theta)a$ begins with $V(\theta_1)aa$ (since any word begins with $a$ after applying $V$), so $\tilde{\theta_1}=V(\theta_1)a$ is the new common prefix. The corresponding argument works for $\tilde{y}$.
		
		As before, there are new cuts appearing from $\dots b|b\dots$ since $V(bb)=ab|ab$. Since the largest subfactor of $b$'s of $uv$ is at the end, let us write $a\theta b=uv=wab^n$ for some word over the alphabet $\{a,b\}$. Since $\theta$ is a palindrome, we have that $\theta$ begins with $b^{n-1}a$. Given any subfactor $b^m$ of $uv$ with $m\geq 2$, the cut $ab^kb|bb^\ell$ becomes
		\begin{equation*}
			\tilde{y}^Tb|a\tilde{x} = \dots aa(ba)^kb|a(ba)^\ell b\dots
		\end{equation*}
		with $k+\ell+2=n$.  In particular that cut satisfies that $\tilde{y}$ begins with $(ab)^kaa$ with $k\leq n-2$ and since $\theta$ begins with $b^{n-1}a$, the word $\tilde{\theta}=V(\theta)a$ begins with $(ab)^{k+1}$ so $\tilde{\theta}_2=(ab)^ka$ is the largest common prefix. Similarly, we have that $\tilde{x}$ begins with $b$ and $\tilde{\theta}=V(\theta)a$ begins with $a$, so the largest common prefix $\tilde{\theta}_1$ is the empty word. In conclusion all cuts of $\tilde{u}^b\tilde{v}\tilde{u}\tilde{v}$ where $\tilde{uv}=V(uv)$, satisfy the properties.
	\end{proof}
	
	\begin{corollary}\label{cor:easy_cor}
		Let $uv\in P$ and write $uv=v_au^b=a\theta b$ where $\theta$ is a palindromic word on $\{a,b\}$. Given any factorization of the bi-infinite word $(uv)^\infty=y^Tb|ax$ different from $\dots uv|uv\dots$ or $(uv)^\infty=x^Ta|by$ different from $\dots v_a|u^b\dots$, then if $w$ is the largest common prefix of $x$ and $y$, then $w$ is an strict prefix of $\theta$ and $x$ begins with $wb$ and $y$ begins with $wa$.
	\end{corollary}
	\begin{remark}
		The previous corollary is easier to prove directly than the proof of \Cref{lem:cuts}. However, \Cref{lem:cuts} gives a stronger conclusion as we will see. 
	\end{remark}

	\begin{corollary}\label{cor:cuts}
		Let $uv\in P$ and write $uv=a\theta b$ where $\theta$ is a palindromic word on $\{a,b\}$. The Markov value of the bi-infinite sequence $(uv)^\infty$ is attained at precisely two positions in its minimal period $uv$, namely the cuts:
		\begin{equation*}
			m\left((uv)^\infty\right)=\lambda(\dots |uv\dots)=\lambda(\dots v_a | u^b \dots).
		\end{equation*}
	\end{corollary}
	
	\begin{proof}
		Clearly the Markov value of $(uv)^\infty=(ab)^\infty$ is being attained at the cut $(ab)^\infty|ab(ab)^\infty$ and $(ab)^\infty a|b(ab)^\infty$. For $uv\in P$, $uv\neq ab$, the \Cref{cor:easy_cor} is sufficient to finish the proof, since for any cut $y^Tb|ax$ different from $\dots uv|uv \dots$ and $x^Ta|by$ different from $\dots v_a|u^b\dots$, one has
		\begin{equation*}
			[a,x]+[0;b,y]<[\overline{a,\theta,b}]+[0;\overline{b,\theta,a}].
		\end{equation*}
		Indeed, since $[a,x]=[a,w,b,\dots]$, $[0;b,y]=[0;b,w,a,\dots]$ and in either case $w$ is a strict prefix of $\theta$, this follows from \Cref{lem:compare}.	
		
		However, \Cref{lem:cuts} gives a stronger conclusion:
		\begin{equation*}
			[a,x]<[\overline{a,\theta,b}] \quad\text{and}\quad [0;b,y]<[0;\overline{b,\theta,a}].
		\end{equation*}
	\end{proof}

	The proof of the following two lemmas are quite similar to \Cref{lem:cuts}, so we will omit the proofs.
	
	\begin{lemma}\label{lem:first_cuts}
		Let $(u,v)\in\overline{P}$ and write $uv=a\theta b$ where $\theta$ is a palindromic word on $\{a,b\}$. Let $x$ and $y$ be words on $\{a,b\}$ defined by any of the following factorizations:
		\begin{itemize}
			\item $u^bvuv=y^Tb|ax$  different from $u^bv|uv=b\theta b|a\theta b$;
			\item $u^bvuv=x^Ta|by$ with length $|xa|<|uv|$;
		\end{itemize}
		Then in any case
		\begin{enumerate}
			\item\label{itm:1} Let $\theta_1$ be the largest prefix common to $x$ and $\theta$. Then either $\theta_1 b$ is a prefix of $x$ and $\theta_1a$ is a prefix of $\theta$, or $x=\theta_1$ is a prefix of $\theta$.
			\item\label{itm:2} Let $\theta_2$ be the largest prefix common to $y$ and $\theta$. Then either $\theta_2 a$ is a prefix of $y$ and $\theta_2 b$ is a prefix of $\theta$, or $y=\theta_2$ is a prefix of $\theta$.
		\end{enumerate}
	\end{lemma}

	\begin{lemma}\label{lem:second_cuts}
		Let $(u,v)\in\overline{P}$ and write $uv=a\theta b$ where $\theta$ is a palindromic word on $\{a,b\}$. Let $x$ and $y$ be words on $\{a,b\}$ defined by any of the following factorizations:
		\begin{itemize}
			\item $u^buvv_a=y^Tb|ax$ with length $|ya|<|uv|$;
			\item $u^buvv_a=x^Ta|by$ different from $u^bv_a|u^bv_a=b\theta a|b\theta b$.
		\end{itemize}
		Then in any case
		\begin{enumerate}
			\item Let $\theta_1$ be the largest prefix common to $x$ and $\theta$. Then either $\theta_1 a$ is a prefix of $x$ and $\theta_1b$ is a prefix of $\theta$, or $x=\theta_1$ is a prefix of $\theta$.
			\item Let $\theta_2$ be the largest prefix common to $y$ and $\theta$. Then either $\theta_2 b$ is a prefix of $y$ and $\theta_2 a$ is a prefix of $\theta$, or $y=\theta_2$ is a prefix of $\theta$.
		\end{enumerate}
	\end{lemma}
	
	The following lemma extends \cite[Theorem 1.5]{Bur+2002} (see also \cite[Theorem 29]{Bombieri}) with a different method, by also considering the Galois conjugates. 
	
	\begin{lemma}\label{lem:position_tildem_attain_max}
		Let $(u,v)\in\overline{P}$ and write $uv=a\theta b$. Then
		\begin{equation}\label{eq:conj1}
			\widetilde{m}\left(u^bv(uv)^\infty\right)=\lambda\left(u^bv|(uv)^\infty\right)=\lambda(b\theta b|(a\theta b)^\infty)=[\overline{a,\theta,b}]+[0;b,\theta,b].
		\end{equation}
		\begin{equation}\label{eq:conj2}
			\widetilde{m}\left((u^bv_a)^\infty\right)=\lambda\left(u^bv_a|(u^bv_a)^\infty\right)=\lambda(b\theta a|(b\theta a)^\infty)=[0;\overline{b,\theta,a}]+[a,\theta,b].
		\end{equation}
	\end{lemma}
	
	\begin{proof}
		We only give the proof of \eqref{eq:conj1}, since the another is analogous.
		
		Let us consider first cuts of the form $u^bv(uv)^\infty=u^bv(uv)^kw_1|w_2(uv)^\infty$ where $w_1,w_2$ are words on $\{1,2\}$ such that $w_1w_2=uv$, $w_1$ is non-empty and $k\geq 0$. If $w_1$ has odd length, then  we have
		\begin{align*}
			\lambda\left(u^bv(uv)^kw_1|w_2(uv)^\infty\right)&=[w_2,\overline{a,\theta,b}]+[0;w_1^T,(b,\theta,a)^k,b,\theta,b] \\
			&<[w_2,\overline{a,\theta,b}]+[0;w_1^T,\overline{b,\theta,a}] \\
			&=\lambda\left((uv)^\infty w_1|w_2(uv)^\infty\right) \leq m\left((uv)^\infty\right).
		\end{align*}
		Since by \eqref{eq:markov_value_first_version} one has that
		\begin{equation*}
			m\left((uv)^\infty\right) = [\overline{a,\theta,b}]+[0;\overline{b,\theta,a}] < [\overline{a,\theta,b}]+[0;\overline{b,\theta,b}],
		\end{equation*}
		the $\widetilde{m}$ value is not attained when $w_1$ has odd length.
		
		If $w_1$ has even length and $k\geq 1$ then
		\begin{align*}
			\lambda\left(u^bv(uv)^kw_1|w_2(uv)^\infty\right)&=[w_2,\overline{a,\theta,b}]+[0;w_1^T,(b,\theta,a)^k,b,\theta,b] \\
			&<[w_2,\overline{a,\theta,b}]+[0;w_1^T,b,\theta,b]=\lambda\left(u^bvw_1|w_2(uv)^\infty\right).
		\end{align*}
		In particular we see that the $\widetilde{m}$ value of $u^bv(uv)^\infty$ is being attained in the prefix $u^bvuv$. Moreover if it is attained at some cut $u^bvw_1|w_2=u^bvuv$, then $w_1$ has even length. In particular, such a cut must have the form $\dotsb b|a\dotsb$, because cuts of the form $22|11, 22|22, 11|11$ all give values at most $[0;2,2,1]+[2;2,2,1]=20/7<35/12=[0;1,1,2,2]+[2;2,1]$ (which is a lower bound for the value of $[\overline{a,\theta,b}]+[0;b,\theta,b]$).
		
		Suppose now that the $\widetilde{m}$ value of $u^bv(uv)^\infty$ is attained at some cut of the form $u^bv(uv)^\infty=x^Ta|by$ with $|xa|<|uv|$ or $u^bv(uv)^\infty=y^Tb|ax$. In any case, \Cref{lem:first_cuts} gives that if $y^Tb|ax$ is different from $b\theta b|a\theta b$, then 
		\begin{equation*}
			[a,x]<[a,\theta] \quad\text{and}\quad [0;b,y]<[b,\theta,b]. 
		\end{equation*}
		In conclusion the value is maximized at the cut $u^bv(uv)^\infty=b\theta b|a\theta a\dots$, which is precisely \eqref{eq:conj1}.
		
	\end{proof}
	
	Finally, we come to the proof of \Cref{thm:equal_value}.

	\begin{proof}[Proof of \Cref{thm:equal_value}]

		It is easy to see that the first two numbers have equal value. Indeed, note that given any positive integers $b_1,\dots,b_m$ one has that $[0;b_1,\dots,b_m,1,1]=[0;b_1,\dots,b_m,2]$. Recall by \eqref{eq:periodic_identity1} we have $((uv)^T)^\infty=u^b(uv)^\infty$. Since the value $\widetilde{m}(x)$ should not be attained at the firsts two positions because $[1;1]+[0;2,2,1,\dots]<[1;1]+[0;2,2,1]=17/7<35/12=[0;1,1,2,2]+[2;2,1]$ (which is a lower bound for the value of $[0;\overline{b,\theta,a}]+[a,\theta,b]$), this shows equality between the first cases. The same argument applies for the last two because of \eqref{eq:periodic_identity2} which gives $v^T((uv)^T)^\infty=u^bv(uv)^\infty$.

		In particular, it suffices to find equality between the second and fourth case. Recall that we can write $uv=a\theta b$ with $\theta^T=\theta$.

		From equalities \eqref{eq:conj1} and \eqref{eq:conj2}, it suffices to show that
		\begin{equation*}
			[\overline{a,\theta,b}]+[0;b,\theta,b]=[0;\overline{b,\theta,a}]+[a,\theta,b].
		\end{equation*}

		In particular letting $\eta=[\overline{a,\theta,b}]=(P+\sqrt{D})/Q$ we have that
		\begin{equation*}
			[\overline{a,\theta,b}]-[0;\overline{b,\theta,a}]=\frac{P+\sqrt{D}}{Q}+\frac{P-\sqrt{D}}{Q}=\frac{2P}{Q}.
		\end{equation*}
		On the other hand we have
		\begin{equation*}
			\eta=[\overline{a,\theta,b}]=\frac{\eta p_n+p_{n-1}}{\eta q_n+q_{n-1}}, \quad \frac{p_n}{q_n}=[a,\theta,b], \quad \frac{p_{n-1}}{q_{n-1}}=[a,\theta,1].
		\end{equation*}
		Therefore we have that $\eta$ satisfies the equation $q_n\eta^2+(q_{n-1}-p_n)\eta-p_{n-1}=0$. This implies that 
		\begin{equation*}
			\frac{2P}{Q}=\frac{p_n-q_{n-1}}{q_n}=[a,\theta,b]-[0;b,\theta,2]=[a,\theta,b]-[0;b,\theta,b]
		\end{equation*}
		where we used that $\beta_{n+1}=q_{n-1}/q_n=[0;a_n,\dots,a_1]$.
		
		Finally we will compute $[\overline{a,\theta,b}]+[0;b,\theta,b]$. For this, let us recall some facts proved in \cite{Bombieri}. If $uv\in P$, then \cite[Theorem 23, (c)]{Bombieri} gives that $M_{uv}$ has the form
		\begin{equation*}
			M_{uv}=\begin{pmatrix}
				3q-q^\prime & \ast \\
				q & q^\prime
			\end{pmatrix}
		\end{equation*}
		In particular, since $uv=a\theta b$, replacing in \eqref{eq:eta_formula} we have that
		\begin{equation*}
			[\overline{a,\theta,b}]=\frac{3q-2q^\prime+\sqrt{\Delta}}{2q},
		\end{equation*}
		where $\Delta=9q^2-4$. In general, one has that $M_{w^T}=M_w^T$, so one has that
		\begin{equation*}
			M_{b\theta a}=M_{(uv)^T}=M_{uv}^T=\begin{pmatrix}
				3q-q^\prime & q \\
				\ast & q^\prime
			\end{pmatrix}
		\end{equation*}
		In particular we have that $[b,\theta,b]=[b,\theta,2]=q/q^\prime$. Therefore from \eqref{eq:markov_value_first_version}
		\begin{equation*}
			[\overline{a,\theta,b}]+[0;b,\theta,b] = \frac{3q-2q^\prime+\sqrt{\Delta}}{2q}+\frac{q^\prime}{q}=\frac{3}{2}+\frac{\sqrt{\Delta}}{2q} = \frac{3+m\left((uv)^\infty\right)}{2}.
		\end{equation*}
		
	\end{proof}

	\section{Similar spectrums and other remarks}\label{sec:remarks}

	Davenport and Schmidt \cite{DavenportSchmidt1970} studied a set called \emph{Dirichlet spectrum} $\mathcal{D}$. For irrational $x$ define
	\begin{equation}\label{def_D}
		D(x)=\limsup_{n\to\infty}\frac{\gamma_{n+1}}{\eta_{n+1}},
	\end{equation} 
	and define the Dirichlet spectrum by
	\begin{equation*}
		\mathcal{D}=\{D(x)<\infty\mid x\in\R\setminus\Q\}.
	\end{equation*}
	Its name comes from its relation to the Dirichlet's approximation theorem, in the sense that if $c(x)=1+\frac{1}{D(x)}$, then
	\begin{equation*}
		c(x)=\sup\left\{c>0\colon \left|x-\frac{p}{q}\right|<\frac{1}{cqQ} \text{ has inf. sol. } (p/q,Q)\in\Q\times\N_{>0},1\leq q\leq Q\right\}
	\end{equation*}
	Similarly with $\mathcal{L}$ and $\mathcal{M}$, the Dirichlet spectrum $\mathcal{D}$ has interesting geometric properties.
	For example, the beginning part of $\mathcal{D}$ is a discrete sequence accumulating at $2+\sqrt{5}$ and $\mathcal{D}$ contains Hall's ray.
	Moreover, the Hausdorff dimension $\dim_H(\mathcal{D}\cap(-\infty,t))$ is continuous with respect to $t$.
	We refer to Section 3.3 of the book \cite{Lim+21} for more details about the spectrum $\mathcal{D}$.
	
	In 1972, Divis \cite{Divis1972} defined a spectrum $\widetilde{\mathcal{D}}$ (he denoted it $\mathcal{M}^*$ in his paper), which is quite related to $\mathcal{D}$. Indeed, if 
	\begin{equation}
		\widetilde{D}(x)=\sup_{n\in\N}\frac{\gamma_{n+1}}{\eta_{n+1}},
	\end{equation}
	then he defined the spectrum
	\begin{equation*}
		\widetilde{\mathcal{D}}=\{\widetilde{D}(x)<\infty\mid x\in\R\setminus\Q\}.
	\end{equation*}
	Notice that the definition of $\widetilde{\mathcal{D}}$ (with respect to $\mathcal{D}$) is quite similar with our spectrum $\widetilde{\mathcal{M}}$ (with respect to $\mathcal{L}$). In \cite{Divis1972}, the author proved that $\widetilde{\mathcal{D}}$ is not closed by proving that the first accumulation point $2+\sqrt{5}$ does not belong to $\widetilde{\mathcal{D}}$ (in contrast observe that the first accumulation point of both $\widetilde{\mathcal{M}}$ and $\mathcal{M}$ is 3 and belongs to them). Also, the beginning part of $\widetilde{\mathcal{D}}$ (it is not homothetic to the beginning $\mathcal{D}$), together with its preimage $\{x\in\R\setminus\Q:\widetilde{D}(x)<2+\sqrt{5}\}$, was fully studied. In fact, it also holds that $\{x\in\R\setminus\Q:\widetilde{D}(x)<2+\sqrt{5}-\varepsilon\}$ is finite for any $\varepsilon>0$ and the preimage $\widetilde{D}^{-1}(t)$ of each $t<2+\sqrt{5}$ contains exactly 0 or 3 elements. Furthermore, the author also proved that $\widetilde{\mathcal{D}}$ contains Hall's ray.
	We expect some of this results to hold for $\widetilde{\mathcal{M}}$, namely, that $\widetilde{\mathcal{M}}$ is not closed and that contains a Hall's ray. 
	
	On the other side, it would be interesting to inquire whether a similar result to \Cref{thm:full_characterization} could be established for the approximation of complex numbers. More precisely, what can we say about numbers $z\in\C$ such that
	\begin{equation*}
		\left|z-\frac{p}{q}\right|<\frac{1}{2|q|^2}
	\end{equation*}
	for only finitely many $p,q\in\Z[i]$ with $p/q\neq z$? For relevant background regarding this point of view, we refer the reader to \cite{AsmusSchmidt} for the choice of constant 2 (see also \cite{Hensley2006} and \cite{Bosma+2012}).

	\bibliographystyle{plain}
	\bibliography{biblography}
	
\end{document}